\documentclass[reqno,10pt]{amsart}

\usepackage{geometry}
\usepackage{graphicx}
\usepackage{epstopdf}

\usepackage{amsmath,amssymb}
\usepackage{bm}
\usepackage{mathtools}
\usepackage{stmaryrd}
\usepackage{mathrsfs}
\usepackage{euscript}
\usepackage{amsfonts}

\usepackage{amsthm}
\usepackage{empheq}
\usepackage{cases}
\numberwithin{equation}{section}

\usepackage{caption}
\usepackage{subcaption}
\usepackage{booktabs}
\usepackage{array}
\usepackage{tabularx}
\usepackage{multirow}
\usepackage{xcolor}

\usepackage[numbers,sort]{natbib}
\usepackage[colorlinks=true, linkcolor=blue, citecolor=blue, urlcolor=cyan]{hyperref}

\usepackage{enumitem}
\usepackage{float}

\newtheorem{theorem}{Theorem}[section]
\newtheorem{lemma}[theorem]{Lemma}
\newtheorem{proposition}[theorem]{Proposition}

\theoremstyle{definition}

\newtheorem{example}[theorem]{Example}

\theoremstyle{remark}
\newtheorem{remark}{Remark}

\newtheoremstyle{assumpstyle}
{8pt}
{8pt}
{\normalfont}
{}
{\bfseries}
{.}
{ }
{}

\theoremstyle{assumpstyle}
\newtheorem{assumption}{Assumption}

\title[Analysis of $L^{p+1}$-normalized gradient flow]{Convergence analysis of $L^{p+1}$-normalized gradient flow for action ground state of nonlinear Schr\"odinger equation}

\author[W. Liu]{Wei Liu}
\address{W.~Liu: College of Science, National University of Defense Technology, Changsha, 410073, China}
\email{wl@nudt.edu.cn}

\author[T. Wang]{Tingfeng Wang}
\address{T.~Wang: School of Mathematics and Statistics, Wuhan University, Wuhan, 430072, China}
\email{tingfengwang@whu.edu.cn}

\author[X. Zhao]{Xiaofei Zhao}
\address{X.~Zhao: School of Mathematics and Statistics \& Computational Sciences, Wuhan University, Wuhan, 430072, China}
\email{matzhxf@whu.edu.cn}

\date{}

\begin{document}

\maketitle

\begin{abstract}
This paper presents a rigorous convergence analysis of the $L^{p+1}$-normalized gradient flow with asymptotic Lagrange multiplier (GFALM) method for computing the action ground state of the nonlinear Schr\"odinger equation in the focusing case. First, a general global convergence theory is established for the semi-discrete GFALM scheme, guaranteeing the existence of an accumulation point and a convergent subsequence. Then, under additional non-degeneracy assumptions, a local exponential convergence rate is rigorously proven. This result is further extended to the fully discrete case using a Fourier pseudo-spectral discretization. The analysis is achieved by characterizing the local geometry of the $L^{p+1}$-constrained manifold near the ground state, establishing a quadratic growth property of the energy functional, and deriving a \L{}ojasiewicz-type gradient inequality. Finally, the paper also investigates the exponential convergence of the associated continuous-time gradient flow, providing a theoretical foundation for future numerical discretizations. This work extends existing convergence analyses for energy ground states, addressing the challenges posed by the $L^{p+1}$ constraint, especially the absence of an inner-product structure.

\noindent\textbf{Keywords.} nonlinear Schr\"odinger equation; action ground state; 
normalized gradient flow; $L^{p+1}$ constraint; convergence analysis; full discretization.

\noindent\textbf{AMS(2010) subject classifications.} 65M12, 35K20, 35K35, 35K55, 65Z05
\end{abstract}

\section{Introduction}
The nonlinear Schr\"odinger equation (NLS)  widely applies in quantum
mechanics, nonlinear optics and deep water waves. Its stationary state solution  $\psi(\mathbf{x},t)=\mathrm{e}^{\mathrm{i}\omega t}\phi(\mathbf{x})$ with  
\begin{align*}
	\mathrm i\partial_t\psi(\mathbf x,t)
	= -\frac{1}{2}\Delta \psi(\mathbf x,t) + V(\mathbf x)\,\psi(\mathbf x,t)
	+ \beta|\psi(\mathbf x,t)|^{p-1}\psi(\mathbf x,t), \quad \mathbf x \in \mathbb{R}^{d},\ t>0,
\end{align*}
is relevant to many wave coherent structures that are of importance in applications, and  
the standing wave or solitary wave $\phi(\mathbf x)$ solves the stationary NLS:
\begin{align}
	-\frac{1}{2}\Delta \phi(\mathbf{x})
	+V(\mathbf{x})\,\phi(\mathbf{x})
	+\beta\bigl|\phi(\mathbf{x})\bigr|^{p-1}\phi(\mathbf{x})
	+\omega\,\phi(\mathbf{x})
	=0, \quad  
    \mathbf x \in \mathbb{R}^{d}.
	\label{eq:semilinear-elliptic-rot}
\end{align}
Here $d\in\mathbb{N}_+$, $V(\mathbf{x})$ denotes an external potential function, $\omega\in\mathbb{R}$ is called the chemical potential, $p>1$ is a given exponent, and $\beta\in\mathbb{R}$ measures the strength of the nonlinear self-interaction with $\beta>0$ and $\beta<0$ corresponding to the defocusing and focusing cases, respectively. For $d\ge 3$, $1<p<(d+2)/(d-2)$ is assumed for well-posedness \cite{berestycki1983nonlinearI, pohozaev1965eigenfunctions}. As a matter of fact, under certain conditions, (\ref{eq:semilinear-elliptic-rot}) possesses infinitely many \cite{berestycki1983nonlinearII, strauss1977existence} nontrivial solutions, i.e., $\phi\neq0$. To characterize the stable and physically relevant ones, two kinds of ground states (GSs) have been defined in the literature.

The first kind is the energy GS that minimizes the energy functional $ E(\phi)$ subject to a  constrained mass $m>0$:
\begin{align}\label{energy gs def}
 \min\{E(\phi)\,:\, \|\phi\|_{L^{2}} = m\},\mbox{ where }   E(\phi) := \frac{1}{2}\|\nabla \phi\|_{L^{2}}^{2} + \int_{\mathbb{R}^d} V|\phi|^2\,{\rm d}\mathbf{x}
    + \frac{2\beta}{p+1}\|\phi\|_{L^{p+1}}^{p+1}.
\end{align}
In this case, the chemical potential $
\omega$ in \eqref{eq:semilinear-elliptic-rot} becomes the Lagrange multiplier, 
and very rich studies  
\cite{bao2013mathematical,bao2004computing, faou2018convergence, liu2021normalized, wang2014projection, zhuang2019efficient,altmann2021jmethod, cances2010numerical, chang2007computing, dion2007ground,antoine2017efficient, danaila2010new, danaila2017computation, henning2020sobolev,CDLX2023JCP} have been devoted to the  mathematical analysis and computational methods of such GS. The second kind of GS considers the action functional $S_{\omega}(\phi)$ with a prescribed $\omega\in\mathbb{R}$:
\begin{align}
	S_{\omega}(\phi)
	:= E(\phi)
	+ \omega\|\phi\|_{L^2}^2, \label{eq:action-functional-S}
\end{align}
which is minimized among all nontrivial stationary state solutions. It has a variational characterization via the minimization of $S_{\omega}(\phi)$ under the associated Nerhari constraint, i.e.,
\begin{equation}\label{action gs def org}
    \min\{S_\omega(\phi)\,:\, K_\omega(\phi)=0,\ \phi\neq0\},
\end{equation}
where $K_\omega(\phi)= \frac{1}{2}\|\nabla \phi\|_{L^{2}}^{2} + \int_{\mathbb{R}^d} V|\phi|^2\,{\rm d}\mathbf{x} + \beta\|\phi\|_{L^{p+1}}^{p+1}+\omega\|\phi\|_{L^{2}}^{2}$.
This action GS problem was popularized by Berestycki and Lions \cite{berestycki1983nonlinearI} for the focusing NLS, i.e., $\beta<0$, which has received adequate further investigations on the theoretical side \cite{ardila2021global, berestycki1983nonlinearI, fukuizumi2001stability, fukuizumi2003instability, shatah1985instability,Dovetta}. In particular, it has been understood recently in \cite{Dovetta,Jeanjean} by theories and in \cite{WangChunshan,liu2023computing} by more general numerical experiments that any energy GS for the focusing NLS must also be an action GS, while the reverse is believed to fail \cite{Dovetta}. This indicates that the set of focusing action GS in fact contains a broader class of standing waves for \eqref{eq:semilinear-elliptic-rot}. To numerically compute the action GS, a direct gradient flow of the constrained minimization \eqref{action gs def org} has been utilized in \cite{WangChunshan}, and we in \cite{liu2023computing} by deriving a mathematically equivalent formulation of \eqref{action gs def org}:
\begin{align}
	\min\left\{\, Q(u)= \frac{1}{2}\|\nabla u\|_{L^{2}}^{2} + \int_{\mathbb{R}^d} V|u|^2\, {\rm d}\mathbf{x}  + \omega\|u\|_{L^2}^2\,:\, \|u\|_{p+1}=1 \right\} \label{action gs eqv},
\end{align}
have proposed a simpler scheme named the normalized gradient flow with asymptotic Lagrange multiplier (GFALM). The GFALM preserves the monotonic decay of $Q$ and can efficiently find the action GS. The practical effectiveness of GFALM is so far in the awaiting of theoretical support, and this is going to be the theme of the paper. For corresponding works on the defocusing case, we refer to \cite{liu2023computing,liu2025action,liu2025computing,ChangWenZhao}.

Assuming $\beta<0$ in \eqref{eq:semilinear-elliptic-rot}, this paper aims to develop systematic convergence theories for GFALM in computing the focusing action GS. To this end, we shall first briefly review GFALM from \cite{liu2023computing}, and then for its semi-discretization version, we shall present a general and global convergence theory that guarantees under minor conditions an accumulation point and convergent subsequence of the numerical series produced from the scheme. Subsequently, our main effort will be devoted to rigorously establishing the practically observed exponential convergence rate of GFALM and extending this result to the fully discrete case.

The existing convergence analysis works so far are only devoted to the energy GS \eqref{energy gs def}, e.g.,  \cite{faou2018convergence,henning2020sobolev,henning2023,zhang2022exponential,CLLZ2024SINUM,feng2025ima}, where the analysis can indeed benefit from the mass constraint, thanks to the linear inner-product structure of the $L^{2}$ setup. Our target problem \eqref{action gs eqv} with the $L^{p+1}$ constraint can be viewed as a generalization, and the absence of inner-product structure makes the geometric analysis and algebraic manipulations essentially different from the $L^{2}$ case. To address these challenges, we develop a new local convergence framework tailored to the $L^{p+1}$-constrained manifold: First, we characterize the local geometry of the $L^{p+1}$-sphere, elucidating the differential-geometric structure of the constraint set near GS; Then, under suitable non-degeneracy assumptions, we establish a quadratic growth property of the $Q$ functional in a neighborhood of the GS, which yields an equivalence between the difference in functional values and the squared discrete $H^1$-distance between functions; Building on this, we show the Lyapunov stability of the scheme, and lastly, by deriving a discrete \L{}ojasiewicz-type gradient inequality, we obtain the exponential convergence of the  GFALM towards its accumulation point. The counterpart of exponential convergence analysis is also established in the end for the continuous gradient flow of \eqref{action gs eqv}, giving rise to the theoretical possibility for more future design of numerical discretizations.

The rest of the paper is organized as follows. In Section~\ref{sec. 2}, we review the GFALM method and establish its first broadly applicable convergence result. Section~\ref{sec. 3} develops the analytical framework that rigorously proves the optimal exponential convergence rate of the GFALM, and extend this framework to the fully discrete case using Fourier spectral discretization. Section~\ref{sec:convergence-nGF} investigates the convergence for the continuous-time normalized gradient flow. Some conclusions are drawn in Section~\ref{sec. 5}. 
For later use, let us here collect and clarify some notation and conventions: $\langle\cdot,\cdot\rangle$ denotes the standard $L^2$-inner product; $a\lesssim b$ means $|a|\le C\, b$ and $a\gtrsim b$ means $|a|\ge C\, b$, for some constant $C>0$ independent of temporal discretization parameters (such as the time step $\tau$ and the time level $n$); $a\asymp b$ means both $a\lesssim b$ and $b\lesssim a$.

\section{GFALM method and its global convergence} \label{sec. 2}

This section will first briefly review the $L^{p+1}$-formulation of the GS and the GFALM proposed in \cite{liu2023computing}. Then, a basic but broadly applicable global convergence result of GFALM will be established.

\subsection{Preliminary of GFALM method}
Throughout the work, we consider $\beta<0$ in \eqref{eq:semilinear-elliptic-rot} or \eqref{eq:action-functional-S}. For numerical computation and theoretical analysis, we consider the problem on a bounded domain $\Omega \subset \mathbb{R}^d$ with appropriate boundary conditions (e.g., homogeneous Dirichlet or periodic boundary conditions). This truncation is justified since the GS under consideration exhibits decay at infinity and  the truncation error becomes negligible by choosing $\Omega$ sufficiently large. 
\begin{assumption}\label{assump:potential-condition}
	Let $ 1 < p < \frac{d+2}{d-2} $ for $ d \geq 3 $ and $ 1 < p < \infty $ for $ d = 1, 2 $. Assume that $V\in L^{\infty}(\Omega)$ and 
	$ V(\mathbf{x}) \geq 0 $ for all $ \mathbf{x} \in \Omega $.
\end{assumption}
For simplicity, we denote $X = H^{1}(\Omega)$. For all $p$ satisfying Assumption~\ref{assump:potential-condition}, $X$ is (compactly) embedded into $ L^{p+1}(\Omega)$, and the related functionals $ S_{\omega} $, $K_{\omega}$ and $Q$ are well-defined in $H^1(\Omega)$. When $V\in L^{\infty}(\Omega)$, the linear operator $-\frac{1}{2}\Delta + V $ has a countably infinite discrete set of eigenvalues with the smallest eigenvalue $\lambda_0$ being simple \cite{GT}: 
\begin{align}
    \lambda_0 := \inf_{u\in X, \, \|u\|_{L^2}=1 }\left( \frac{1}{2}\|\nabla u\|_{L^2}^2 + \int_{\Omega}V|u|^2\,{\rm d}\mathbf{x} \right). \label{eq:omega-lambda-0-def}
\end{align} 
For the existence of GS, one needs to confine the $\omega$ as:
\begin{assumption}\label{assump:omega-conditions}
	Assume that $\omega$ satisfies $\omega > -\lambda_0$.
\end{assumption}

Under Assumptions~\ref{assump:potential-condition}\&\ref{assump:omega-conditions}, the GS $\phi_{g}$ of \eqref{eq:action-functional-S} is guaranteed and can be equivalently described as the following $L^{p+1}$-normalization formulation (detailed in \cite{liu2023computing}): $\phi_{g} = \left( Q(u_{g}) / (-\beta) \right)^{1/(p-1)}u_{g}$, where
\begin{align}
	u_{g}\in \arg\min\{\, Q(u):\ u\in \mathcal S_{p+1}\, \} \label{eq:ug-definition}
\end{align}
with $\mathcal{S}_{p+1} := \{u\in X:\, \|u\|_{L^{p+1}}=1\}$ and a quadratic energy functional
\begin{align}
	Q(u) := \frac{1}{2}\|\nabla u\|_{L^{2}}^{2} + \int_{\Omega} V|u|^2\, {\rm d}\mathbf{x}  + \omega\|u\|_{L^2}^2. \label{eq:action-functional-Q}
\end{align}
We shall still refer to $u_g$ from \eqref{eq:ug-definition} as GS for simplicity. It in some sense could be viewed as a generalization of the energy GS with prescribed mass, whose analysis however is quite different, as we shall see later due to the absence of inner-product structure.

To solve \eqref{eq:ug-definition}, it is natural to consider a normalized gradient flow. With the variation of the quadratic functional $Q$,
\begin{align*}
	\nabla Q(u) := \frac{\delta Q}{\delta \bar{u}}(u,\bar{u}) =  \left( -\frac{1}{2}\Delta + V + \omega \right)u =: \mathcal{A}u,
\end{align*}
the continuous normalized gradient flow associated with \eqref{eq:ug-definition} reads: for $u=u(\mathbf{x},t)$,
\begin{align}
	\partial_{t}u = -\mathcal{A}u+ \lambda(u)\,|u|^{p-1}u, \quad t\ge 0,\ \mathbf{x}\in\Omega, \label{eq:constrained-gradient-flow}
\end{align}
where $\lambda(u) = {\left\langle \mathcal{A}u, |u|^{p-1} u\right\rangle}\big/{\|u\|_{L^{2p}}^{2p}}$  is the Lagrange multiplier to enforce the constraint $\|u\|_{L^{p+1}} = 1$. 
In such a way, it is straightforward to deduce from \eqref{eq:ug-definition} that 
\begin{align}
	\frac{{\rm d}}{{\rm d} t}\|u(\cdot,t)\|_{L^{p+1}}^{p+1} = 0, \quad 
	\frac{{\rm d}}{{\rm d} t}Q(u(\cdot,t)) = -2\|\partial_{t} u(\cdot,t)\|_{L^{2}}^{2} \le 0,\quad \forall t\geq0. 
	\label{eq:lpplus1-conservation-energy-decay}
\end{align}
The continuous-time gradient flow \eqref{eq:constrained-gradient-flow} forms an effective foundation for numerical computations, while its direct discretizations could suffer from the complexity of $\lambda(u)$ on practical efficiency \cite{liu2025computing} and numerical analysis. Further discussion of \eqref{eq:constrained-gradient-flow} will be presented in Section~\ref{sec:convergence-nGF}.

By the Euler--Lagrange equation of \eqref{eq:ug-definition}, we have at any stationary point $\widetilde{u}\in \mathcal{S}_{p+1}$,
\begin{align*}
	\mathcal{A}\widetilde{u} = \widetilde{\lambda}(\widetilde{u}) |\widetilde{u}|^{p-1}\widetilde{u} \quad \mbox{with }\; \widetilde{\lambda}(\widetilde{u}) := \frac{Q(\widetilde{u})}{\|\widetilde{u}\|_{L^{p+1}}^{p+1}}=Q(\widetilde{u}). 
\end{align*}
Motivated by this observation, \cite{liu2023computing} proposed using the asymptotic Lagrange multiplier in \eqref{eq:constrained-gradient-flow} followed by a piecewise normalization step in time, which leads to 
\begin{align}
	\begin{cases}
		\partial_{t}u = -\left( \mathcal{A} - \widetilde{\lambda}(u(\cdot,t_{n}))\,|u(\cdot,t_{n})|^{p-1} \right)u, & t\in \left[ t_{n}, t_{n+1} \right), \\[0.2em]
		u(\mathbf{x},t_{n}) = u(\mathbf{x},t_{n}^{+}) = \dfrac{u(\mathbf{x},t_{n}^{-})}{\|u(\mathbf{x},t_{n}^{-})\|_{L^{p+1}}}, & n\ge 0,
	\end{cases} \label{eq:GFALM}
\end{align}
where $t_n=n\tau$ ($n\in\mathbb{N}$) with a selected time step size $\tau>0$, and $u(\mathbf{x},0)$ is some initial guess for the minimizer of \eqref{eq:ug-definition}. This simplification of the multiplier yields considerable flexibility in the time discretization of \eqref{eq:GFALM}.

A backward-forward Euler discretization has been further suggested for \eqref{eq:GFALM} in \cite{liu2023computing}, forming the following GFALM scheme on a box domain $\Omega\subset\mathbb{R}^d$ with homogeneous Dirichlet or periodic boundary conditions: 
\begin{subequations}\label{eq:GFALM-semi-disc.}
	\begin{align}
		&\frac{1}{\tau}\bigl(\widetilde{u}^{n+1} - u^{n}\bigr) = -\widetilde{\mu}^{n+1}, \quad  
		\widetilde{\mu}^{n+1} = \Bigl( -\frac{1}{2}\Delta + \alpha \Bigr)\widetilde{u}^{n+1} + \Bigl( V + \omega - \widetilde{\lambda}(u^{n})|u^{n}|^{p-1} - \alpha \Bigr)u^{n}, \label{eq:GFALM-semi-suba.} \\
		&u^{n+1} = \widetilde{u}^{n+1} \big/ \|\widetilde{u}^{n+1}\|_{L^{p+1}}. \label{eq:GFALM-semi-subb.}
	\end{align}
\end{subequations}
Here $\alpha\geq0$ is a stabilization parameter, which is chosen to such that 
\begin{align}
	\alpha \ge \frac{1}{2}\max\Big\{0,\ \operatorname{ess\,sup}_{\mathbf{x}\in \Omega}\big( V(\mathbf{x})+\omega \big)\Big\}+\frac{1}{2}, \label{eq:alphan-condition}
\end{align}
to ensure the unconditional decay of the energy $Q$ stated below.

\begin{lemma}[Energy-decaying property]\label{lem:energy_decaying}
	Under Assumptions~\ref{assump:potential-condition}\&\ref{assump:omega-conditions}, and with $\alpha$ satisfying \eqref{eq:alphan-condition}, the GFALM scheme \eqref{eq:GFALM-semi-disc.} possesses the unconditional energy-decaying property for $Q(u)$: $\forall\,\tau>0$,
    \begin{align}
		Q(u^{n+1}) - Q(u^{n}) 
        \le -\frac{\tau^2\|\nabla\widetilde{\mu}^{n+1}\|_{L^2}^{2}+\tau(\tau+4)\|\widetilde{\mu}^{n+1}\|_{L^2}^{2}}{2\|\widetilde{u}^{n+1}\|_{L^{p+1}}^{2}}
        \le -\frac{\tau^2\|\widetilde{\mu}^{n+1}\|_{H^1}^{2}}{2\|\widetilde{u}^{n+1}\|_{L^{p+1}}^{2}},
        \quad n\geq0. \label{eq:energy_decaying}
	\end{align}
\end{lemma}
\begin{proof}
Following the proof of Theorem~2.4 in \cite{liu2023computing}, the GFALM scheme \eqref{eq:GFALM-semi-disc.} implies the identity
	\begin{align*}
		Q(\widetilde{u}^{n+1})
        &=\widetilde{\lambda}(u^{n})\int_{\Omega}|u^n|^{p-1}|\widetilde{u}^{n+1}|^{2}d\mathbf{x}-\frac12\|\nabla(\widetilde{u}^{n+1}-u^n)\|_{L^2}^{2} \\
        &\qquad -\int_{\Omega}\left(\frac{2}{\tau}+2\alpha-V-\omega+\widetilde{\lambda}(u^{n})|u^n|^{p-1}\right)|\widetilde{u}^{n+1}-u^n|^{2}d\mathbf{x}.
	\end{align*}
    Applying the facts that $\widetilde{\lambda}(u^{n})=Q(u^n)>0$, $\int_{\Omega}|u^n|^{p-1}|\widetilde{u}^{n+1}|^{2}d\mathbf{x}\leq\|u^{n}\|_{L^{p+1}}^{p-1}\|\widetilde{u}^{n+1}\|_{L^{p+1}}^2=\|\widetilde{u}^{n+1}\|_{L^{p+1}}^2$, and $2\alpha-V-\omega+\widetilde{\lambda}(u^{n})|u^n|^{p-1}\geq\alpha\geq\frac12$ a.e. $\mathbf{x}\in\Omega$, results in
	\begin{align*}
		Q(\widetilde{u}^{n+1})\leq Q(u^{n})\|\widetilde{u}^{n+1}\|_{L^{p+1}}^2-\frac12\|\nabla(\widetilde{u}^{n+1}-u^n)\|_{L^2}^{2}-\left(\frac12+\frac{2}{\tau}\right)\|\widetilde{u}^{n+1}-u^n\|_{L^2}^{2}.
	\end{align*}
    The assertion is obtained immediately by substituting $\widetilde{u}^{n+1}-u^n=-\tau\widetilde{\mu}^{n+1}$ \eqref{eq:GFALM-semi-suba.} and noting that $Q(u^{n+1})={Q(\widetilde{u}^{n+1})}/{\|\widetilde{u}^{n+1}\|_{L^{p+1}}^{2}}$.
\end{proof}

\subsection{A basic unconditional convergence result} \label{sec.2.2}
The capability of GFALM \eqref{eq:GFALM-semi-disc.} to effectively compute the GS of \eqref{eq:action-functional-S} has already been demonstrated in \cite{liu2023computing}, which is not yet theoretically analyzed. Here, we provide the first convergence result, which is basic but general and global, for the semi-discrete GFALM \eqref{eq:GFALM-semi-disc.} with any step size $\tau>0$. 

\begin{theorem}[Accumulation point \& unconditional convergent subsequence]\label{thm:convergent-subsequences}
	Suppose that Assumptions~\ref{assump:potential-condition}\&\ref{assump:omega-conditions} hold and $\alpha$ satisfies \eqref{eq:alphan-condition}. 
    Furthermore, for $d\geq 3$, we additionally assume $1<p\leq d/(d-2)$. Let  $u^{0}\in H^{1}(\Omega) \cap \mathcal{S}_{p+1}$. Then, for every fixed $\tau>0$:
\begin{enumerate}[label=(\roman*)]
    \item the solution $\left\{ u^{n} \right\}_{n=0}^{\infty}$ of \eqref{eq:GFALM-semi-disc.} has a strongly convergent subsequence $\left\{ u^{n_{j}} \right\}_{j=0}^{\infty}$ in $H^{1}(\Omega)$; 
    \item every accumulation point $u^{\star}\in \mathcal{S}_{p+1}$ of $\left\{ u^{n} \right\}_{n=0}^{\infty}$ in $H^{1}(\Omega)$ satisfies the Euler-Lagrange equation of \eqref{eq:ug-definition} in the weak sense, i.e.,
	\begin{align}\label{eq:Euler-LagrangeEq-weakform}
		\left\langle \mathcal{A}u^{\star}, v \right\rangle = \widetilde{\lambda}(u^{\star})\left\langle |u^{\star}|^{p-1}u^{\star}, v \right\rangle, \quad \forall v\in H^{1}(\Omega),
	\end{align}
    and all accumulation points possess the same Lagrange multiplier and energy: $\widetilde{\lambda}(u^{\star})=Q(u^\star)=\inf_{n\geq0}Q(u^n)$.
\end{enumerate}
\end{theorem}

The proof of Theorem~\ref{thm:convergent-subsequences} will be given after establishing the following lemmas. 
The first is a tool for bounding the $H^1$-norm by the functional $Q$.

\begin{lemma}\label{lem:h1-control-by-Q}
	Under Assumptions \ref{assump:potential-condition}\&\ref{assump:omega-conditions}, there exists a constant $C_0>0$ depending only on $\lambda_0$ and $\omega$ such that
	\begin{align}
		Q(v)\ge C_0\|v\|_{H^{1}}^{2},\quad \forall v\in X. \label{eq:Q-control-H1}
	\end{align}
\end{lemma}
\begin{proof}
    Since $V(\mathbf{x})\geq0$, it is clear from \eqref{eq:omega-lambda-0-def} that $\lambda_{0} \ge 0$. If $\omega>0$, the assertion is straightforward:
    \[ Q(v) =\left\langle \mathcal{A}v, v \right\rangle \geq\frac{1}{2}\|\nabla v\|_{L^{2}}^{2}+\omega \|v\|_{L^{2}}^{2}\geq \min\left\{\frac12,\omega\right\}\|v\|_{H^{1}}^{2}. \]
    If $-\lambda_{0} < \omega \leq 0$, we perform the following decomposition for $\mathcal{A}$:
    	\begin{align*}
    		\mathcal{A} =\frac{\lambda_{0}+\omega}{2\lambda_{0}}\mathcal{A}_{0}+ \left(\frac{\lambda_{0}-\omega}{2\lambda_{0}}\mathcal{A}_{0} + \omega I\right),
    	\end{align*}
    with $\mathcal{A}_{0}= -\frac{1}{2}\Delta + V$ and $I$ the identity. Therefore,
    	\begin{align*}
    		Q(v) =\left\langle \mathcal{A}v, v \right\rangle 
            \geq \frac{\lambda_{0}+\omega}{4\lambda_{0}}\|\nabla v\|_{L^{2}}^{2}+\frac{\lambda_{0}+\omega}{2} \|v\|_{L^{2}}^{2}
            \geq (\lambda_{0}+\omega)\min\left\{ \frac{1}{4\lambda_{0}},\frac12 \right\}\|v\|_{H^1}^{2},
    	\end{align*}
    which yields the assertion since $\lambda_0+\omega>0$. 
\end{proof}

Denote $\mu^{n}:= \mathcal{A}u^{n} - \widetilde{\lambda}(u^{n})|u^{n}|^{p-1}u^{n}=\mathcal{A}u^{n} - Q(u^{n})|u^{n}|^{p-1}u^{n}$, then subtracting $\mu^{n}$ from \eqref{eq:GFALM-semi-suba.} gives
\begin{align*}
	\widetilde{\mu}^{n+1} - \mu^{n} = -\frac{1}{2}\Delta\left( \widetilde{u}^{n+1} - u^{n} \right)  + \alpha\left( \widetilde{u}^{n+1} - u^{n} \right) 
	= \frac{\tau}{2}\Delta\widetilde{\mu}^{n+1} - \tau\alpha\widetilde{\mu}^{n+1}. 
\end{align*}
Rearranging the terms above yields the following elliptic resolvent equation (with the same boundary conditions as in \eqref{eq:GFALM-semi-disc.}):
\begin{align}
	\left( 1 + \tau\alpha \right) \widetilde{\mu}^{n+1} - \frac{\tau}{2}\Delta\widetilde{\mu}^{n+1} =  \mu^{n}. \label{eq:ellipse-equation-mu}
\end{align}
Based on \eqref{eq:ellipse-equation-mu}, we then establish for the numerical solution some uniform-in-time bounds.

\begin{lemma}\label{lem:uniform-estimates}
    Under Assumptions \ref{assump:potential-condition}\&\ref{assump:omega-conditions}, for the GFALM scheme \eqref{eq:GFALM-semi-disc.} with $\alpha$ satisfying \eqref{eq:alphan-condition}, the following uniform upper bounds hold: 
    \begin{align*}
        \|u^n\|_{H^1}\lesssim 1,\quad 
        \tau\|\widetilde{\mu}^{n+1}\|_{H^1}\lesssim 1,\quad
        \|\widetilde{u}^{n+1}\|_{H^1}\lesssim 1, \quad \forall n\in\mathbb{N}.
    \end{align*}
\end{lemma}
\begin{proof}
    By Lemma~\ref{lem:h1-control-by-Q} and the energy decay \eqref{eq:energy_decaying}, we see that $\|u^{n}\|_{H^{1}}^2\lesssim Q(u^n)\leq Q(u^0)$ for all $n\geq0$. To bound $\tau\|\widetilde{\mu}^{n+1}\|_{H^{1}}$, we take the $L^2$ inner product of \eqref{eq:ellipse-equation-mu} with $\widetilde{\mu}^{n+1}$, yielding that
	\begin{align*}
		\left( 1 + \tau\alpha \right) \|\widetilde{\mu}^{n+1}\|_{L^{2}}^{2} + \frac{\tau}{2}\|\nabla \widetilde{\mu}^{n+1}\|_{L^{2}}^{2} 
        &= \left\langle \mu^{n}, \widetilde{\mu}^{n+1} \right\rangle \\
		&= \big| \langle \mathcal{A}u^{n} - \widetilde{\lambda}(u^{n})|u^{n}|^{p-1}u^{n}, \widetilde{\mu}^{n+1}\rangle \big| \\
        &\lesssim \|u^{n}\|_{H^{1}}\|\widetilde{\mu}^{n+1}\|_{H^{1}}+Q(u^{n})\|u^{n}\|_{L^{p+1}}^p\|\widetilde{\mu}^{n+1}\|_{L^{p+1}} \\
        &\lesssim \|\widetilde{\mu}^{n+1}\|_{H^{1}}.
	\end{align*}
    Noting that $\alpha\geq 1/2$, we get $\tau\|\widetilde{\mu}^{n+1}\|_{H^1}\lesssim 1$. Furthermore, we can find
	\begin{align*}
		\|\widetilde{u}^{n+1}\|_{H^{1}}
        =\|u^{n}-\tau\widetilde{\mu}^{n+1}\|_{H^{1}}
		\le \|u^{n}\|_{H^{1}} + \tau\|\widetilde{\mu}^{n+1}\|_{H^{1}} 
        \lesssim 1,
	\end{align*}
	which completes the proof.
\end{proof}

\begin{lemma}\label{lem:tilde-u-nplus1-gap-1}
	Under Assumptions \ref{assump:potential-condition}\&\ref{assump:omega-conditions}, for the GFALM scheme \eqref{eq:GFALM-semi-disc.} with $\alpha$ satisfying \eqref{eq:alphan-condition}, the following uniform lower bound holds:
	\begin{align*}
		\|\widetilde{u}^{n+1}\|_{L^{p+1}} \gtrsim \frac{1}{1+\sqrt{Q(u^n)-Q(u^{n+1})}}\gtrsim 1, \quad \forall n\in\mathbb{N}.
	\end{align*}
\end{lemma}
\begin{proof}
    By the triangle inequality,
    \begin{align}\label{eq:tilde-u-nplus1-gap-1-Step1}
        1=\left\|u^n\right\|_{L^{p+1}}
        =\left\|\widetilde{u}^{n+1}+\tau\widetilde{\mu}^{n+1}\right\|_{L^{p+1}}
        \leq \left\|\widetilde{u}^{n+1}\right\|_{L^{p+1}}+\tau\left\|\widetilde{\mu}^{n+1}\right\|_{L^{p+1}}.
    \end{align}
    The Sobolev embedding $H^1(\Omega)\hookrightarrow L^{p+1}(\Omega)$ and the energy decay \eqref{eq:energy_decaying} state that
    \begin{align}\label{eq:tilde-u-nplus1-gap-1-Step2}
        \tau\left\|\widetilde{\mu}^{n+1}\right\|_{L^{p+1}}
        \lesssim \tau\left\|\widetilde{\mu}^{n+1}\right\|_{H^1}
        \lesssim \left\|\widetilde{u}^{n+1}\right\|_{L^{p+1}}\sqrt{Q(u^n)-Q(u^{n+1})}.
    \end{align}
    Noting that $Q(u_g)\leq Q(u^{n+1})\leq Q(u^n)\leq Q(u^0)$, the combination of \eqref{eq:tilde-u-nplus1-gap-1-Step1} and \eqref{eq:tilde-u-nplus1-gap-1-Step2} completes the proof immediately.
\end{proof}

\begin{proof}[Proof of Theorem~\ref{thm:convergent-subsequences}]
By \eqref{eq:GFALM-semi-suba.}, we have 
\begin{align*}
	(1+\tau\alpha)\widetilde{u}^{n+1} -\frac{\tau}{2}\Delta\widetilde{u}^{n+1} = u^{n} +\tau \left(\alpha- V - \omega \right)u^{n} + \tau\widetilde{\lambda}(u^{n})|u^{n}|^{p-1}u^{n} =: g(u^{n}).
\end{align*}
Substituting \eqref{eq:GFALM-semi-subb.} into it gives
\begin{align}\label{eq:u{n+1}-elliptic-eq}
	(1+\tau\alpha) u^{n+1} -\frac{\tau}{2}\Delta u^{n+1} = \|\widetilde{u}^{n+1}\|_{L^{p+1}}^{-1}g(u^{n}).
\end{align}
Taking the $L^2$ inner product of it with $-\Delta u^{n+1}$ yields
\begin{align}
	(1+\tau\alpha) \|\nabla u^{n+1}\|_{L^{2}}^{2} + \frac{\tau}{2}\|\Delta u^{n+1}\|_{L^{2}}^{2} \leq \|\widetilde{u}^{n+1}\|_{L^{p+1}}^{-1}\|g(u^{n})\|_{L^{2}} \|\Delta u^{n+1}\|_{L^{2}}.\label{eq:H2-estinate-step1}
\end{align}
From the definition of $g(u^{n})$ and the Sobolev embedding $H^{1}(\Omega) \hookrightarrow L^{2p}(\Omega)$ (which holds under the assumption on $p$), we obtain
\begin{align}
	\|g(u^{n})\|_{L^{2}} \lesssim (1+\tau)\|u^{n}\|_{L^{2}} + \tau \|u^{n}\|_{L^{2p}}^{p} \lesssim (1+\tau)\|u^{n}\|_{L^{2}} + \tau \|u^{n}\|_{H^{1}}^{p}. \label{eq:H2-estinate-step2}
\end{align}
Inserting \eqref{eq:H2-estinate-step2} into \eqref{eq:H2-estinate-step1} and then employing the uniform bound of $\|\widetilde{u}^{n+1}\|_{L^{p+1}}^{-1}$ provided by Lemma~\ref{lem:tilde-u-nplus1-gap-1}, we get
$\frac{\tau}{2}\|\Delta u^{n+1}\|_{L^{2}}^{2} \lesssim \left[ (1+\tau)\|u^{n}\|_{L^{2}} + \tau \|u^{n}\|_{H^{1}}^{p} \right] \|\Delta u^{n+1}\|_{L^{2}}$,  which leads to
\begin{align*}
	\|\Delta u^{n+1}\|_{L^{2}} \lesssim (1+\tau^{-1})\|u^{n}\|_{L^{2}} + \|u^{n}\|_{H^{1}}^{p}.
\end{align*}
It means that for all $n\in \mathbb{N}^{+}$,  $u^{n}\in H^{2}(\Omega)$ and $\|u^{n}\|_{H^{2}} \leq C_{\tau}\|u^{n}\|_{H^{1}}$, where $C_{\tau}>0$ depends on $\tau$. The uniform $H^{1}$-bound of $\{ u^{n} \}_{n=0}^{\infty}$ in Lemma~\ref{lem:uniform-estimates} implies the uniform $H^{2}$-bound of $\{ u^{n} \}_{n=1}^{\infty}$ (for every fixed $\tau>0$). By the Banach-Alaoglu theorem, there exists a weakly convergent subsequence $\{u^{n_{j}}\}_{j=0}^{\infty}$ in $H^2(\Omega)$, and the compact embedding of $H^{2}(\Omega)\hookrightarrow  H^{1}(\Omega)$ implies the strong convergence of $\{u^{n_{j}}\}_{j=0}^{\infty}$ in the $H^{1}$-norm (see also \cite{chen2025second}). This establishes the first assertion of Theorem~\ref{thm:convergent-subsequences}.

Now we turn to the second assertion of Theorem~\ref{thm:convergent-subsequences}. Let $u^\star\in \mathcal{S}_{p+1}$ be an accumulation point of $\{u^{n}\}_{n=0}^{\infty}$ with $\{u^{n_{j}}\}_{j=0}^{\infty}$ a subsequence converging strongly to $u^\star$ in the $H^{1}$-norm. 
Denoting 
\begin{align*}
	a\left( u,v \right) = \int_{\Omega}{ \left[ \frac{1}{2}\nabla u \cdot \nabla \bar{v} + \left( V + \omega \right)u \, \bar{v} - \widetilde{\lambda}(u)|u|^{p-1}u \, \bar{v} \right] }\,\text{d}\mathbf{x},
\end{align*}
we can find that for all $v\in H^{1}(\Omega)$, 
\begin{align}
	a\left( u^{n_{j}},v \right) - a\left( u^{\star},v \right) 
    =&\ \int_{\Omega}{ \left[ \frac{1}{2}\nabla \left( u^{n_{j}} - u^{\star} \right) \cdot\nabla \overline{v} + \left( V + \omega \right)\left( u^{n_{j}} - u^{\star} \right) \, \overline{v} \right] }\mathrm{d}\mathbf{x} \nonumber \\
	&\ - \left( Q(u^{n_{j}})-Q(u^{\star})\right) \int_{\Omega}{ |u^{n_{j}}|^{p-1}u^{n_{j}}\,\overline{v} }\,\mathrm{d}\mathbf{x}  \nonumber \\
	&\ - Q(u^{\star})\int_{\Omega}{ \left(|u^{n_{j}}|^{p-1}u^{n_{j}}-|u^{\star}|^{p-1}u^{\star} \right)\,\overline{v} }\,\mathrm{d}\mathbf{x}.
    \label{eq:auj-austar}
\end{align}
It is easy to observe that
\begin{align*}
\left|Q(u^{n_{j}})-Q(u^{\star}) \right| 
&= \left|\operatorname{Re}\int_{\Omega}\left[\frac12\nabla (u^{n_{j}}+u^{\star})\cdot\nabla \overline{(u^{n_{j}}-u^{\star})}+ (V+\omega) (u^{n_{j}}+u^{\star})\overline{(u^{n_{j}}-u^{\star})} \right]\mathrm{d}\mathbf{x}\right| \\
&\lesssim \left(\|u^{n_{j}}\|_{H^1}+\|u^{\star}\|_{H^1}\right) \|u^{n_{j}}-u^{\star}\|_{H^1} \lesssim \|u^{n_{j}}-u^{\star}\|_{H^1}.
\end{align*}
Note that 
\begin{align*}
|u^{n_{j}}|^{p-1}u^{n_{j}}-|u^{\star}|^{p-1}u^{\star}
=\int_0^1\left(\frac{p+1}{2}|u_{\rho}|^{p-1}(u^{n_{j}}-u^{\star}) + \frac{p-1}{2}|u_{\rho}|^{p-3}u_{\rho}^2\,\overline{(u^{n_{j}}-u^{\star})}\right)\mathrm{d}\rho,
\end{align*}
where $u_{\rho}:=\rho u^{n_{j}}+(1-\rho)u^{\star}$, and 
$\|u_{\rho}\|_{L^{p+1}}\leq \rho\|u^{n_{j}}\|_{L^{p+1}}+(1-\rho)\|u^{\star}\|_{L^{p+1}}=1$, $\forall\,\rho\in[0,1]$. We can obtain 
\begin{align*}
\left| \int_{\Omega}{ \left(|u^{n_{j}}|^{p-1}u^{n_{j}}-|u^{\star}|^{p-1}u^{\star} \right)\,\overline{v} }\,\mathrm{d}\mathbf{x} \right|
&\leq p\max_{0\leq\rho\leq 1}\|u_{\rho}\|_{L^{p+1}}^{p-1}\left\|u^{n_{j}}-u^{\star}\right\|_{L^{p+1}}\|v\|_{L^{p+1}} \\
&\lesssim \|u^{n_{j}}-u^{\star}\|_{H^1}\|v\|_{H^{1}}.
\end{align*}
Substituting these estimates into \eqref{eq:auj-austar} yields
\begin{align*}
\big| a\left( u^{n_{j}},v \right) - a\left( u^{\star},v \right) \big|
    \lesssim &\  \|u^{n_{j}} - u^{\star}\|_{H^{1}} \|v\|_{H^{1}}
    +\big| Q(u^{n_{j}})-Q(u^{\star})\big| \|u^{n_{j}}\|_{L^{p+1}}^p \|v\|_{L^{p+1}} \\
 &\    + \left|\int_{\Omega}{ \left(|u^{n_{j}}|^{p-1}u^{n_{j}}-|u^{\star}|^{p-1}u^{\star} \right)\,\overline{v} }\,\mathrm{d}\mathbf{x}\right| \\
 \lesssim &\ \|u^{n_{j}} - u^{\star}\|_{H^{1}} \|v\|_{H^{1}}.
\end{align*}
The $H^1$-strong convergence of the subsequence $\{u^{n_{j}}\}_{j=0}^{\infty}$ implies that, 
for all $v\in H^{1}(\Omega)$, 
\begin{align*}
	\lim_{j\rightarrow\infty} a\left( u^{n_{j}},v \right) = a\left( u^{\star},v \right).
\end{align*}
On the other hand, noting that $\mu^{n}= \mathcal{A}u^{n} - \widetilde{\lambda}(u^{n})|u^{n}|^{p-1}u^{n}$, we find from \eqref{eq:ellipse-equation-mu} and the the energy decay \eqref{eq:energy_decaying} that for any $v\in H^1(\Omega)$, 
\begin{align*}
	\left\langle\mu^{n},v\right\rangle
    &=(1+\tau\alpha)\left\langle \widetilde{\mu}^{n+1},v\right\rangle+\frac{\tau}{2}\left\langle \nabla \widetilde{\mu}^{n+1},\nabla v\right\rangle 
     \leq \left(1+\alpha\tau\right)\|\widetilde{\mu}^{n+1}\|_{H^1}\|v\|_{H^1}\\
    &\lesssim \left(\alpha+\tau^{-1}\right)\|\widetilde{u}^{n+1}\|_{L^{p+1}}\sqrt{Q(u^{n}) - Q(u^{n+1})}\, \|v\|_{H^1}\xrightarrow{n\to\infty} 0.
\end{align*}
This means that $\left| a\left( u^{n_{j}}, v \right) \right| = \left| \left\langle \mu^{n_{j}}, v \right\rangle \right| \xrightarrow{j\to\infty} 0$, 
and so
\begin{align*}
	a\left( u^{\star}, v \right) = \int_{\Omega}{ \left[ \frac{1}{2}\nabla u^{\star} \cdot\nabla \bar{v} + \left( V + \omega \right)u^{\star} \, \bar{v} - \widetilde{\lambda}(u^{\star})|u^{\star}|^{p-1}u^{\star} \, \bar{v} \right] }\,\text{d}\mathbf{x} = 0, \quad \forall\, v\in H^{1}(\Omega),
\end{align*}
which confirms \eqref{eq:Euler-LagrangeEq-weakform}. Finally, since the energy sequence $\{Q(u^n)\}$ is non-increasing and bounded from below, it has a unique limit, i.e.,
\[ \widetilde{\lambda}(u^{\star})=Q(u^{\star})=\lim_{j\to\infty}Q(u^{n_j})=\lim_{n\to\infty}Q(u^n)=\inf_{n\geq0}Q(u^n). \]
\end{proof}

\begin{remark}
    When $d\geq 3$, for the case of Dirichlet boundary conditions with some regularity assumption on the boundary $\partial\Omega$, we can get rid of the additional assumption $p\leq d/(d-2)$ made in Theorem~\ref{thm:convergent-subsequences} by employing the $W^{2,q}$ estimate for the linear elliptic equation \eqref{eq:u{n+1}-elliptic-eq} (see, e.g., \cite[Sect.~9.6]{GT}):
    \begin{align}\label{eq:Lq-estimate}
        \left\|u^{n+1}\right\|_{W^{2,q}}\leq C(\tau)\left\|\widetilde{u}^{n+1}\right\|_{L^{p+1}}^{-1}\left\|g(u^n)\right\|_{L^q}, \quad q:=\frac{2d}{p(d-2)}\in \left(\frac{2d}{d+2},\frac{2d}{d-2}\right),
    \end{align}
    where $C(\tau)>0$ depends on $\tau$ but is independent of $n$. Actually, since $\left\|\widetilde{u}^{n+1}\right\|_{L^{p+1}}^{-1}\lesssim 1$ by Lemma~\ref{lem:tilde-u-nplus1-gap-1} and $\left\|g(u^n)\right\|_{L^q}\lesssim (1+\tau)\left\|u^n\right\|_{L^q}+\tau \left\|u^n\right\|_{L^{pq}}^p
    \lesssim (1+\tau)\left\|u^n\right\|_{H^1}+\tau \left\|u^n\right\|_{H^1}^p\lesssim 1+\tau$ by Lemma~\ref{lem:uniform-estimates} and continuous embeddings $H^1(\Omega)\hookrightarrow L^{q}(\Omega)$ and $H^1(\Omega)\hookrightarrow L^{pq}(\Omega)$, \eqref{eq:Lq-estimate} establishes the boundedness of sequence $\{u^n\}_{n=1}^{\infty}$ in $W^{2,q}(\Omega)$ for fixed $\tau>0$. Then, the strongly convergent subsequence in $H^1$ follows from the compact embedding $W^{2,q}(\Omega)\hookrightarrow H^1(\Omega)$.
\end{remark}

\section{Exponential convergence of GFALM} \label{sec. 3}

This section is devoted to rigorously establishing an exponential convergence rate of the semi-discrete GFALM scheme \eqref{eq:GFALM-semi-disc.}. For simplicity of notation, we will assume that functions within the discussion are all real-valued. Extensions to complex-valued cases could also be done in the same manner, which will be discussed in the end. 

\subsection{Geometry near GS}

To conduct a rigorous local convergence analysis of the semi-discrete GFALM scheme \eqref{eq:GFALM-semi-disc.}, we first characterize the local geometry structure of the constraint manifold $\mathcal{S}_{p+1}$ near $u_g$. Specifically,  we introduce local coordinates around $u_g$, impose a crucial non-degeneracy assumption, and establish a local equivalence between the energy error and the $H^1$ error. 

Define the constraint functional $G(u) :=  \frac{2}{p+1} \big( \|u\|_{L^{p+1}}^{p+1}-1 \big)$, where $G(u)=0$ is equivalent to $u\in \mathcal{S}_{p+1}$. It is easy to verify that $G(u)$ has first- and second-order G\^ateaux derivatives:
\begin{align*}
    &\mathcal{D}G(u)[v]=2\langle |u|^{p-1}u,v\rangle,\quad
    \mathcal{D}^2G(u)[w,v]=2p\langle |u|^{p-1}w,v\rangle,\quad \forall u,v,w\in L^{p+1}(\Omega).
\end{align*}
Moreover,  by virtue of the following lemma,  $G(u)$ is of class $C^2(L^{p+1}(\Omega),\mathbb{R})$ (see also \cite[Proposition 1.12]{willem1996minimax}).
\begin{lemma}[Theorem A.2 of \cite{willem1996minimax}] \label{lem:f2}
    On a bounded domain $\Omega$, for $1<p<\infty$, the mapping $u\mapsto|u|^{p-1}$ from $L^{p+1}(\Omega)$ to $L^{\frac{p+1}{p-1}}(\Omega)$ is continuous.
\end{lemma}

Clearly, the tangent space at $u \in \mathcal{S}_{p+1}= \{u \in X : G(u)=0\}$ reads
\begin{align*}
	\mathcal{T}_{u}
	:= \big\{ \xi \in X :\; \mathcal{D}G(u)[\xi]
	= 2\left\langle |u|^{p-1}u, \xi \right\rangle = 0 \big\}.
\end{align*}
Assume that $\lambda_g$ is the corresponding Lagrange multiplier of GS $u_g$, i.e.,
\begin{align}
	\mathcal{A}u_{g} = \lambda_{g}\, |u_{g}|^{p-1}u_{g}.
	\label{eq:euler-lagrange-ug}
\end{align}
We now introduce a local coordinate near the GS $u_{g}$:
\begin{align*}
	\chi:&\ \mathbb{R}\times \mathcal{T}_{u_{g}}\to X,\quad (r,\xi)\mapsto (1+r)\,u_{g}+\xi. 
\end{align*}
Clearly, $\chi$ is a bijection onto the image, with the inverse 
\begin{align*}
	\chi^{-1}:\ u\in X\mapsto 
    \bigl(r(u),\xi(u)\bigr)
	=\Bigl(\langle |u_{g}|^{p-1}u_{g},u\rangle^{} - 1,\,
	u-\langle |u_{g}|^{p-1}u_{g},u\rangle^{}\,u_{g}\Bigr)\in\mathbb{R}\times \mathcal{T}_{u_{g}}.
\end{align*}
Denote the neighborhood region of $u_{g}$ as $U_{\delta}(u_{g}) := \left\{ u\in X \,:\, \|u - u_{g}\|_{H^{1}} \leq \delta \right\}$. The following lemma shows that the normalization $\|u\|_{L^{p+1}}=1$ can uniquely determine $r$ as a $C^2$ function of $\xi\in\mathcal{T}_{u_{g}}$. The proof is based on the implicit function theorem and some basic calculations given in Appendix~\ref{appdx:proof-lem:implicit-function-normalization}.
\begin{lemma}\label{lem:implicit-function-normalization}
	Under Assumption~\ref{assump:potential-condition}, there exists some $\delta>0$ such that for any $u\in U_{\delta}(u_{g})$ admitting the representation $u=\chi(r,\xi)$, the implicit relation $\|u\|_{L^{p+1}}=1$ yields a unique $C^2$ map
	\begin{align*}
		r:\ \mathcal{T}_{u_{g},\delta}\;\to\;\mathbb{R}_{\delta},\qquad \xi\mapsto r(\xi),
	\end{align*}
	where $\mathcal{T}_{u_{g},\delta} \subset \mathcal{T}_{u_{g}}$ and $\mathbb{R}_{\delta}\subset\mathbb{R}$ are neighborhoods of zero, satisfying $r(0)=0$ and
	\begin{align}
		|r(\xi)|\;\lesssim\; \|\xi\|_{L^{p+1}}^2 \lesssim \|\xi\|_{H^{1}}^{2}. \label{eq:estimate-r-disc.}
	\end{align}
\end{lemma}

Under assumptions of Lemma~\ref{lem:implicit-function-normalization}, we also obtain that, when $\delta$ is small enough, for $u\in U_{\delta}(u_{g})\cap \mathcal{S}_{p+1}$, 
\begin{align*}
	\|u-u_{g}\|_{H^{1}} = \|r u_{g} + \xi\|_{H^{1}} \leq |r| \cdot \|u_{g}\|_{H^{1}} + \|\xi\|_{H^{1}} \lesssim \|\xi\|_{H^{1}},
\end{align*}
and, conversely, there exists $C>0$ such that
\begin{align*}
	\|u-u_{g}\|_{H^{1}} =& \|r u_{g} + \xi\|_{H^{1}} \geq -|r| \cdot \|u_{g}\|_{H^{1}} + \|\xi\|_{H^{1}} \geq -C\|\xi\|_{H^{1}}^{2} + \|\xi\|_{H^{1}} \\
	\geq& \left( 1 - C\|\xi\|_{H^{1}} \right) \|\xi\|_{H^{1}}.
\end{align*}
For $\delta$ sufficiently small, it follows that $\|\xi\|_{H^{1}} \lesssim \|u-u_{g}\|_{H^{1}}$. In other words, $u - u_{g}$ is also equivalent to $\xi$ in the discrete $H^{1}$-norm, i.e., 
\begin{align}
\|u-u_{g}\|_{H^{1}} \asymp \|\xi\|_{H^{1}},\quad \forall\, u\in U_{\delta}(u_{g})\cap \mathcal{S}_{p+1}.    \label{eq:u-ug-approx-xi}
\end{align}

For technique analysis reason, we need $u_{g}$ to be a non-degenerate minimizer of $Q$ on $\mathcal{S}_{p+1}$. Here we impose a standard non-degeneracy (coercivity) condition on the second-order variation of the Lagrangian 
\begin{align}
	L(u,\lambda) := Q(u) - \lambda\,G(u), \label{eq:lagrangian-Q-G1}
\end{align}
restricted to the tangent space. Define the linear operator $\mathcal{L}v := \mathcal{A}v - p\,\lambda_{g}|u_{g}|^{p-1} v, \  \forall\, v\in X.$

\begin{proposition}[Uniqueness]\label{prop:coercive-second-variation}
	Assume that there exists a constant $c>0$ such that
	\begin{align}
		\mathcal{D}_{u}^{2} L(u_{g},\lambda_{g})[\xi, \xi] \;\ge\; c\,\|\xi\|_{H^{1}}^{2},
		\qquad \forall\, \xi \in \mathcal{T}_{u_{g}}. \label{eq:coercivity-second-variation}
	\end{align}
	Then, for a sufficiently small $\delta>0$, there holds
	\begin{align}
		Q(u) - Q(u_{g}) \asymp \|u-u_{g}\|_{H^{1}}^{2}, \quad \forall\, u \in U_{\delta}(u_{g}) \cap \mathcal{S}_{p+1}, \label{eq:energy-gap-quadratic}
	\end{align}
    so that the constrained minimization problem \eqref{eq:ug-definition} admits $u_{g}$ as the unique minimizer in $U_{\delta}(u_{g})\cap \mathcal{S}_{p+1}$.
\end{proposition}
\begin{proof}
	It is direct to see that 
	\begin{align}
		Q(u) - Q(u_{g})
		=&\, \left\langle \mathcal{A}u, u \right\rangle - \left\langle \mathcal{A}u_{g}, u_{g} \right\rangle
		= \left\langle \mathcal{A}\left( u + u_{g} \right), \left( u - u_{g} \right) \right\rangle \nonumber \\
		=&\, \left\langle \mathcal{A}\left( u - u_{g} \right), u - u_{g} \right\rangle + 2\left\langle \mathcal{A}u_{g}, u - u_{g} \right\rangle \nonumber \\
		=&\, \left\langle \mathcal{A}\left( u - u_{g} \right), u - u_{g} \right\rangle + 2\lambda_{g}\left\langle |u_{g}|^{p-1}u_{g}, u - u_{g} \right\rangle. \label{eq:gap-Qu-Qu-star}
	\end{align}
    By expanding the constraint functional, we have
	\begin{align}
		G(u) &=G(u_g)+\mathcal{D}G(u_g)[u-u_g]+\int_{0}^{1}{ (1-\rho)\mathcal{D}^2G(u_{g}+\rho(u - u_{g}))[u-u_g,u-u_g] }\, \mathrm{d}\rho \nonumber \\
		&=G(u_g)+ 2\left\langle |u_{g}|^{p-1}u_{g}, u - u_{g} \right\rangle + p\left\langle |u_{g}|^{p-1}\left( u - u_{g} \right), u - u_{g} \right\rangle \nonumber \\
        &\quad\; + 2p\int_{0}^{1}{ (1-\rho) \left\langle \left( \left|u_{g}+\rho(u - u_{g})\right|^{p-1} - |u_{g}|^{p-1} \right)(u - u_{g}), u - u_{g}  \right\rangle }\, \mathrm{d}\rho. \label{eq:expanding-G-at-ug}
	\end{align}
    By H\"older's inequality, and combining with Lemma~\ref{lem:f2}, we have
    \begin{align}
    	&\left|  \int_{0}^{1}{ (1-\rho) \left\langle \left( \left|u_{g}+\rho(u - u_{g})\right|^{p-1} - |u_{g}|^{p-1} \right)(u - u_{g}), u - u_{g}  \right\rangle }\, \mathrm{d}\rho \right| \nonumber \\
        \lesssim&\, \left\| \left|u_{g}+\rho(u - u_{g})\right|^{p-1} - |u_{g}|^{p-1} \right\|_{L^{\frac{p+1}{p-1}}} \left\| u-u_{g} \right\|_{L^{p+1}}^{2} = o\!\left( \left\| u-u_{g} \right\|_{L^{p+1}}^{2} \right). \label{eq:high-order-esti}
    \end{align}
    Since $G(u)=G(u_g)=0$, it follows from \eqref{eq:expanding-G-at-ug} and \eqref{eq:high-order-esti} that
	\begin{align*}
		\left\langle |u_{g}|^{p-1}u_{g}, u - u_{g} \right\rangle
		= -\frac{p}{2}\left\langle |u_{g}|^{p-1}\left( u - u_{g} \right), u - u_{g} \right\rangle
		+ o\!\left( \|u - u_{g}\|_{L^{p+1}}^{2} \right).
	\end{align*}
    Substituting this into \eqref{eq:gap-Qu-Qu-star} yields
	\begin{align}
		Q(u) - Q(u_{g})
		=&\, \left\langle \mathcal{A}\left( u - u_{g} \right), u - u_{g} \right\rangle
		- p\,\lambda_{g}\left\langle |u_{g}|^{p-1}\left( u - u_{g} \right), u - u_{g} \right\rangle
		+ o\!\left( \|u - u_{g}\|_{L^{p+1}}^{2} \right) \nonumber\\
		=&\, \left\langle \mathcal{L}\left( u - u_{g} \right), u - u_{g} \right\rangle
		+ o\!\left( \|u - u_{g}\|_{H^{1}}^{2} \right). \label{eq:Qu-Qug-Step1}
	\end{align}

    Note that the first- and second-order G\^ateaux derivatives of the Lagrangian \eqref{eq:lagrangian-Q-G1} read
	\begin{align*}
		&\mathcal{D}_{u}L(u,\lambda)[v_{1}]
		= \mathcal{D}Q(u)[v_{1}] - \lambda\,\mathcal{D}G(u)[v_{1}]
		= 2\,\left\langle \mathcal{A}u - \lambda |u|^{p-1}u, v_{1} \right\rangle, \\[0.25em]
		&\mathcal{D}_{u}^{2}L(u_{g},\lambda_{g})[v_{2}, v_{1}]
		= 2\,\left\langle \mathcal{A}v_{2} - p\,\lambda_{g}|u_{g}|^{p-1} v_{2}, v_{1} \right\rangle
		= 2\left\langle \mathcal{L}v_{2}, v_{1} \right\rangle, \quad \forall v_1,v_2\in X.
	\end{align*}
    By \eqref{eq:coercivity-second-variation}, we have 
    \begin{align}\label{eq:linearized-operator-coercivity}
        \left\langle \mathcal{L}\xi, \xi \right\rangle \gtrsim \|\xi\|_{H^{1}}^{2},\quad \forall\, \xi\in \mathcal{T}_{u_{g}}.
    \end{align}
    On the other hand,  for all $\xi\in \mathcal{T}_{u_{g}}$,
    \begin{align*}
		\left\langle \mathcal{L}\xi, \xi \right\rangle 
        =\left\langle \mathcal{A}\xi, \xi \right\rangle -p\lambda_g\left\langle  |u_{g}|^{p-1}\xi, \xi \right\rangle 
        \lesssim \|\xi\|_{H^{1}}^{2}+p\lambda_g\|u_g\|_{L^{p+1}}^p\|\xi\|_{L^{p+1}}^2
        \lesssim \|\xi\|_{H^{1}}^{2}. 
	\end{align*}
    Consequently,
    \begin{align}
        \langle\mathcal{L}\,\xi,\xi\rangle \asymp \|\xi\|_{H^{1}}^{2}, \quad \forall\, \xi\in \mathcal{T}_{u_{g}}. \label{eq:Qu-Qug-Step2}
    \end{align}
    Moreover, for all $u\in U_{\delta}(u_{g}) \cap \mathcal{S}_{p+1}$ with $\delta$ chosen as in Lemma~\ref{lem:implicit-function-normalization}, we can write
	\begin{align*}
		u-u_{g} = r(\xi)u_{g} + \xi \qquad\mbox{with}\quad r(\xi)=\mathcal{O}\!\left(\|\xi\|_{H^1}^{2}\right), \quad \xi\in\mathcal{T}_{u_{g}},
	\end{align*}
    and deduce that 
	\begin{align}
		\langle\mathcal{L}\,( u-u_{g} ),u-u_{g}\rangle
		=&\, \langle\mathcal{L}\,\xi,\xi\rangle  + 2r\,\langle\mathcal{L}\,u_{g},\xi\rangle  + r^{2}\langle\mathcal{L}\,u_{g},u_{g}\rangle
        = \langle\mathcal{L}\,\xi,\xi\rangle + \mathcal{O}\!\left( \|\xi\|_{H^{1}}^{4} \right), \label{eq:Qu-Qug-Step3}
	\end{align}
    where $\langle \mathcal{L}\,u_{g}, \xi \rangle
		= \big\langle (\mathcal{A}u_{g} - p\,\lambda_{g}|u_{g}|^{p-1}u_{g}),\xi\big\rangle
		= (1-p)\,\lambda_{g}\,\langle |u_{g}|^{p-1}u_{g},\xi\rangle
		= 0$. 
    Substituting \eqref{eq:Qu-Qug-Step2}-\eqref{eq:Qu-Qug-Step3} into \eqref{eq:Qu-Qug-Step1}, and applying $\|u-u_{g}\|_{H^{1}}^{2}\asymp\|\xi\|_{H^{1}}^{2}$ stated in \eqref{eq:u-ug-approx-xi}, we can conclude that for all $u\in U_{\delta}(u_{g}) \cap \mathcal{S}_{p+1}$ with $\delta>0$ small enough,
	\begin{align*}
		Q(u) - Q(u_{g})
		=\langle\mathcal{L}\,\xi,\xi\rangle + o\!\left( \|\xi\|_{H^{1}}^{2} \right)
        \asymp \|\xi\|_{H^{1}}^{2}
        \asymp \|u-u_{g}\|_{H^{1}}^{2},
	\end{align*}
    which completes the proof of Proposition~\ref{prop:coercive-second-variation}.
\end{proof}

\subsection{Exponential convergence rate}

Now we are ready to present the convergence rate result for the semi-discrete GFALM scheme \eqref{eq:GFALM-semi-disc.}.

\begin{theorem}[Main Result]\label{thm:exponential-convergence-ground-state}
    Suppose that Assumptions~\ref{assump:potential-condition}\&\ref{assump:omega-conditions} hold and that the stabilization parameter $\alpha$ satisfies \eqref{eq:alphan-condition}. Let $u_{g}$ be the GS defined by \eqref{eq:ug-definition} and suppose that it satisfies the non-degeneracy condition in Proposition~\ref{prop:coercive-second-variation}. Then, there exists a constant $\delta_{0}>0$ such that for any initial data $u^{0} \in U_{\delta_{0}}(u_{g}) \cap \mathcal{S}_{p+1}$ and any $\tau>0$, the GFALM scheme \eqref{eq:GFALM-semi-disc.} converges to the GS at the rate:
    \begin{align}\label{eq:exponential-convergence-to-ground-state}
        \|u^{n} - u_{g}\|_{H^{1}} \le C\,\mathrm{e}^{-a n \tau/(1+\tau)}, \qquad \forall\, n \in \mathbb{N},
    \end{align}
    where $C, a > 0$ are constants independent of $n$ and $\tau$.
\end{theorem}

We then turn to the proof of Theorem \ref{thm:exponential-convergence-ground-state}. As a preparation, we give some estimations derived from the energy decaying property.

\begin{lemma}\label{lem:esti-by-energe-gap}
    Assume that Assumptions~\ref{assump:potential-condition}\&\ref{assump:omega-conditions} hold, and let $\alpha$ satisfy \eqref{eq:alphan-condition}. For the solution of \eqref{eq:GFALM-semi-disc.}, we have
    \begin{align*}
    	\tau\left\| \widetilde{\mu}^{n+1} \right\|_{H^{1}} \lesssim  \sqrt{ Q(u^{n}) - Q(u^{n+1}) }, \quad \left\| u^{n+1} - u^{n} \right\|_{H^{1}} \lesssim \sqrt{ Q(u^{n}) - Q(u^{n+1}) },\quad \forall n\in\mathbb{N}.
    \end{align*}
\end{lemma}
\begin{proof}
    From the decay of energy \eqref{eq:energy_decaying}, we have
    \begin{align*}
        \tau^{2} \left\| \widetilde{\mu}^{n+1} \right\|_{H^{1}}^{2} \leq 2\left\| \widetilde{u}^{n+1} \right\|_{L^{p+1}}^{2}\left( Q(u^{n}) - Q(u^{n+1}) \right).
    \end{align*}
    Combining the continuous embedding $H^{1} \hookrightarrow L^{p+1}$ with the uniform bounds provided by Lemma~\ref{lem:uniform-estimates}, we obtain
    \begin{align*}
    	\tau^{2} \left\| \widetilde{\mu}^{n+1} \right\|_{H^{1}}^{2} \lesssim \left\| \widetilde{u}^{n+1} \right\|_{H^{1}}^{2} \left(  Q(u^{n}) - Q(u^{n+1})\right) \lesssim Q(u^{n}) - Q(u^{n+1}),
    \end{align*}
    which gives the first estimate. Furthermore, from the normalization step \eqref{eq:GFALM-semi-subb.}, we have
    \begin{align*}
    	\left\| u^{n+1} - u^{n} \right\|_{H^{1}} 
        =&\, \left\|u^{n+1}- \widetilde{u}^{n+1}+\widetilde{u}^{n+1} - u^{n} \right\|_{H^{1}} \\
        \leq&\,  \left\| u^{n+1} \right\|_{H^{1}} \left| \left\| \widetilde{u}^{n+1} \right\|_{L^{p+1}}-1\right| + \left\| \widetilde{u}^{n+1} - u^{n} \right\|_{H^{1}} \\
        =&\, \left\| u^{n+1} \right\|_{H^{1}} \left| \left\| \widetilde{u}^{n+1} \right\|_{L^{p+1}}-\left\| u^n \right\|_{L^{p+1}}\right| + \left\| \widetilde{u}^{n+1} - u^{n} \right\|_{H^{1}} \\
        \lesssim&\, \left\| \widetilde{u}^{n+1} - u^{n} \right\|_{H^{1}}
        =\tau\left\| \widetilde{\mu}^{n+1} \right\|_{H^{1}} \lesssim \sqrt{ Q(u^{n}) - Q(u^{n+1}) },
    \end{align*}
    which completes the proof.
\end{proof}

We can then show that the numerical solution $ u^{n} $ for all $n\in \mathbb{N}$ can stay within a local range of $ U_{\delta}(u_{g}) \cap \mathcal{S}_{p+1} $, which is crucial for convergence proof. 

\begin{lemma}[Lyapunov stability]\label{lem:local-stability-ground-state}
	Let the assumptions in Theorem~\ref{thm:exponential-convergence-ground-state} hold and let $\delta>0$ be a sufficiently small constant that meets the requirements of Lemma~\ref{lem:implicit-function-normalization} and Proposition~\ref{prop:coercive-second-variation}. Then, there exists a  $\delta_{0}>0$ such that for all $u^{0} \in U_{\delta_{0}}(u_{g})\cap \mathcal{S}_{p+1}$, the solution of \eqref{eq:GFALM-semi-disc.} with initial guess $u^{0}$ satisfies
	\begin{align*}
		\|u^{n}-u_{g}\|_{H^{1}} \leq \delta, \qquad \forall\, n\in \mathbb{N}.
	\end{align*}
\end{lemma}
\begin{proof}
    From \eqref{eq:energy-gap-quadratic}, there exist constants $0 < c_1 \le c_2$ such that
    \begin{align}\label{eq:energy-gap-quadratic-c1c2}
    	c_1\left\| u - u_{g} \right\|_{H^{1}}^{2} \leq Q(u) - Q(u_{g}) \leq c_2\left\| u - u_{g} \right\|_{H^{1}}^{2}, \quad \forall\, u \in U_{\delta}(u_{g}) \cap \mathcal{S}_{p+1}.
    \end{align}
    By Lemma~\ref{lem:esti-by-energe-gap}, 
    \begin{align*}
    	\left\| u^{n+1} - u^{n} \right\|_{H^{1}} \lesssim \sqrt{Q(u^{n}) - Q(u^{n+1})} \lesssim \sqrt{Q(u^{0}) - Q(u_{g})}\lesssim \|u^0-u_g\|_{H^1},\quad \forall\,n\in\mathbb{N},
    \end{align*}
    and we can say for a constant $c_{3}>0$ that
    \begin{align}\label{eq:un-step-less-u0ug}
        \left\| u^{n+1} - u^{n} \right\|_{H^{1}} \leq c_{3} \left\| u^{0} - u_{g} \right\|_{H^{1}}.
    \end{align}
    Fix $\delta_{0}>0$ such that $\delta_{0}<\frac{\delta}{2}\min\left\{ \sqrt{\frac{c_1}{c_2}}, \frac{1}{c_3}\right\}$. We now proceed to show that the assertion of the lemma holds for this $\delta_{0}$. We argue by contradiction, supposing that there exists $n_{1}\in \mathbb{N}^{+}$ with $\|u^{n_{1}} - u_{g}\|_{H^{1}} > \delta$. Since, by \eqref{eq:un-step-less-u0ug}, $\left\| u^{n+1} - u^{n} \right\|_{H^{1}} \leq c_{3}\delta_{0}<\delta/2$, there exists $0 \neq n_{2} < n_{1}$ such that $\delta/2 <\|u^{n_{2}} - u_{g}\|_{H^{1}} \leq \delta$. Using \eqref{eq:energy-gap-quadratic-c1c2} at $u^{0}$ and $u^{n_{2}}$, respectively, one obtains
    \begin{align*}
    	Q(u^{n_{2}}) -Q(u_{g}) &\geq  c_1\|u^{n_{2}}-u_g\|_{H^1}^2
        > c_1\delta^{2}/4, \\
    	Q(u^{0}) -Q(u_{g}) &\leq c_2\left\| u^{0} - u_{g} \right\|_{H^{1}}^{2} \leq c_2\delta_{0}^{2}.
    \end{align*}
 Due to $\delta_{0}<\frac{\delta}{2} \sqrt{\frac{c_1}{c_2}}$, we have $c_2\delta_0^2<c_1\delta^2/4$ and thus $Q(u^{0}) < Q(u^{n_{2}})$, which contradicts the energy decay in Lemma~\ref{lem:energy_decaying}.
\end{proof}

Under the non-degeneracy condition, we can establish the following \L{}ojasiewicz-type gradient inequality near the GS as the last tool.

\begin{lemma}[\L{}ojasiewicz-type gradient inequality]\label{lem:Lojasiewicz-Inequality}
	Suppose that Assumptions~\ref{assump:potential-condition}\&\ref{assump:omega-conditions} hold,  and the non-degeneracy condition in Proposition~\ref{prop:coercive-second-variation} hold for the GS $u_g$. Then, there exists a sufficiently small $\delta>0$ such that
	\begin{align}
		Q(u) - Q(u_{g}) \lesssim \|\mu(u)\|_{H^{-1}}^{2},
		\quad \forall\, u\in U_{\delta}(u_{g})\cap \mathcal{S}_{p+1}, \label{eq:Lojasiewicz-Inequality}
	\end{align}
	where $\mu(u):=\mathcal{A}u-Q(u)|u|^{p-1}u$. 
\end{lemma}
\begin{proof}
    Note that $\mu(u_g)=\mathcal{A}u_g-Q(u_g)|u_g|^{p-1}u_g=\mathcal{A}u_g-\lambda_g|u_g|^{p-1}u_g=0$. For every $u\in U_{\delta}(u_{g})\cap \mathcal{S}_{p+1}$, denoting $v:=u-u_g$, we have
    \begin{align*}
        \left\langle \mu(u), v\right\rangle
        &=\left\langle \mu(u)-\mu(u_g), v\right\rangle \\
        &=\left\langle \mathcal{A}v,v\right\rangle-\left\langle Q(u)|u|^{p-1}u-Q(u_g)|u_g|^{p-1}u_g, v\right\rangle \\
        &=\left\langle \mathcal{A}v,v\right\rangle - \left(Q(u)-Q(u_g)\right)\left\langle |u|^{p-1}u, v\right\rangle-\lambda_g\left\langle |u|^{p-1}u-|u_g|^{p-1}u_g, v\right\rangle.
    \end{align*}
    Applying the facts $\left(Q(u)-Q(u_g)\right)\left\langle |u|^{p-1}u, v\right\rangle
        \lesssim\|v\|_{H^1}^2\|u\|_{L^{p+1}}^p\|v\|_{L^{p+1}}
        =\mathcal{O}\left(\|v\|_{H^1}^3\right)$ and
    \begin{align}
        \left\langle |u|^{p-1}u-|u_g|^{p-1}u_g, v\right\rangle 
        &= \int_{0}^{1}{ p\,\left\langle |u_{g}+\rho\, v|^{p-1}v, v \right\rangle }\, \mathrm{d}\rho \nonumber\\
        &= p\left\langle |u_g|^{p-1}v, v\right\rangle + \int_{0}^{1}{ p\,\left\langle \left( |u_{g}+\rho\, v|^{p-1} - |u_{g}|^{p-1} \right)v, v \right\rangle }\, \mathrm{d}\rho \nonumber\\
        &= p\left\langle |u_g|^{p-1}v, v\right\rangle + \mathcal{O}\left( \left\| |u_{g}+\rho\, v|^{p-1} - |u_{g}|^{p-1} \right\|_{L^{\frac{p+1}{p-1}}} \left\| v \right\|_{L^{p+1}}^{2} \right) \nonumber\\
        &= p\left\langle |u_g|^{p-1}v, v\right\rangle + o\,( \left\| v \right\|_{H^{1}}^{2} ), \label{eq:high-order-esti-2}
    \end{align}
    we obtain
    \begin{align*}
        \left\langle \mu(u), v\right\rangle
        &=\left\langle \mathcal{A}v,v\right\rangle-p\lambda_g\left\langle |u_g|^{p-1}v, v\right\rangle+o\left(\|v\|_{H^1}^2\right) 
        =\left\langle \mathcal{L}v,v\right\rangle+o\left(\|v\|_{H^1}^2\right).
    \end{align*}
    From the definition of $H^{-1}$-norm and the coercivity property \eqref{eq:linearized-operator-coercivity} of the operator $\mathcal{L}$, we have  for some constant $c>0$ that
    \begin{align*}
		\|\mu(u)\|_{H^{-1}}
		\ge \frac{\langle \mathcal{L}v, v\rangle}{\|v\|_{H^{1}}} + o\left(\|v\|_{H^1}\right)
		\ge c\|v\|_{H^{1}}+ o\left(\|v\|_{H^1}\right).
	\end{align*}
    Using the local equivalence between $\|v\|_{H^{1}}^2$ and $Q(u) - Q(u_{g})$, we find when $\delta>0$ is small enough, 
   \begin{align*}
		\|\mu(u)\|_{H^{-1}}
		\ge \frac{c}{2}\|v\|_{H^{1}}
        \gtrsim \sqrt{Q(u)-Q(u_g)},
	\end{align*}
	which proves the assertion.
\end{proof}

Noting that $\mu^n=\mathcal{A}u^n-Q(u^n)|u^n|^{p-1}u^n=\mu(u^n)$, the following result is a direct application of Lemma~\ref{lem:local-stability-ground-state} and Lemma~\ref{lem:Lojasiewicz-Inequality}:
\begin{lemma}\label{lem:Lojasiewicz-Gradient-Inequality}
	Under the conditions of Theorem~\ref{thm:exponential-convergence-ground-state}, there exists some $\delta_{0}>0$ such that for all $u^{0}\in U_{\delta_{0}}(u_{g})\cap \mathcal{S}_{p+1}$,
	\begin{align}
		Q(u^{n}) - Q(u_{g}) \le C_{Lo}\,\|\mu^{n}\|_{H^{-1}}^{2},
		\quad \forall\, n\in\mathbb{N}, \label{eq:Lojasiewicz-Gradient-Inequality}
	\end{align}
	with $C_{Lo}>0$ a constant independent of $n,\tau$.
\end{lemma}

\begin{proof}[Proof of Theorem~\ref{thm:exponential-convergence-ground-state}] 
From the elliptic equation \eqref{eq:ellipse-equation-mu}, we can  derive that
\begin{align*}
	\left\langle \mu^{n}, v \right\rangle \lesssim
    &\, \left( 1 + \tau \alpha \right)\left\|\widetilde{\mu}^{n+1}\right\|_{L^{2}}\|v\|_{L^{2}} + \tau \left\| \nabla \widetilde{\mu}^{n+1} \right\|_{L^{2}}\left\| \nabla v \right\|_{L^{2}} \lesssim \left( \left\|\widetilde{\mu}^{n+1}\right\|_{L^{2}} + \tau \left\| \widetilde{\mu}^{n+1} \right\|_{H^{1}} \right)\|v\|_{H^{1}},
\end{align*}
which by noting the definition of $H^{-1}$-norm can give $ \|\mu^{n}\|_{H^{-1}} \lesssim \left\|\widetilde{\mu}^{n+1}\right\|_{L^{2}} + \tau \left\| \widetilde{\mu}^{n+1} \right\|_{H^{1}}$. By Lemma~\ref{lem:energy_decaying} and the $H^{1}$-bound of $\widetilde{u}^{n+1}$ in Lemma~\ref{lem:uniform-estimates}, we have
\begin{align*}
	\tau\|\mu^{n}\|_{H^{-1}}^{2}
	&\lesssim \tau\|\widetilde{\mu}^{n+1}\|_{L^{2}}^{2} + \tau^{3}\left\| \widetilde{\mu}^{n+1} \right\|_{H^{1}}^{2} \\
	&\lesssim \left(1+\tau\right)\left(\tau\|\widetilde{\mu}^{n+1}\|_{L^{2}}^{2} + \tau^{2}\left\| \widetilde{\mu}^{n+1} \right\|_{H^{1}}^{2}\right) \\
    & \lesssim (1+\tau)\left(Q(u^{n}) - Q(u^{n+1})\right).
\end{align*}
Hence, for some constant $c_{D}>0$ satisfying $c_{D}\leq C_{Lo}$, 
\[ Q(u^{n}) - Q(u^{n+1}) \ge c_{D}\frac{\tau}{1+\tau} \|\mu^{n}\|_{H^{-1}}^{2}, \]
which combining with \eqref{eq:Lojasiewicz-Gradient-Inequality} gives
\begin{align*}
	Q(u^{n}) - Q(u^{n+1})
	\ge \frac{c_{D}}{C_{Lo}}\,\frac{\tau}{1+\tau}\big( Q(u^{n}) - Q(u_{g}) \big).
\end{align*}
Denoting $a_{0}=c_{D}/C_{Lo}\in(0,1]$, we then have 
\begin{align}
	Q(u^{n+1}) - Q(u_{g}) \le \left(1-\frac{a_{0}\tau}{1+\tau}\right)\big( Q(u^{n}) - Q(u_{g}) \big),\quad n\in\mathbb{N}. \label{eq:Q-nplus1-Q-ug}
\end{align}
Using \eqref{eq:Q-nplus1-Q-ug} recursively, we find that for every $\tau>0$,
\begin{align*}
	Q(u^{n}) - Q(u_{g})
	\le \left(1-\frac{a_{0}\tau}{1+\tau}\right)^{n}\big( Q(u^{0}) - Q(u_{g}) \big)
	\le \exp\left(-n\frac{a_{0}\tau }{1+\tau}\right)\big( Q(u^{0}) - Q(u_{g}) \big).
\end{align*}
Using the local equivalence $Q(u^{n})-Q(u_{g}) \asymp \|u^{n}-u_{g}\|_{H^{1}}^{2}$ from Proposition~\ref{prop:coercive-second-variation}, we complete the proof of Theorem~\ref{thm:exponential-convergence-ground-state} by setting $a=a_{0}/2$.
\end{proof}

\subsection{Full discretization and its convergence analysis}
Now we look at the case of fully discretized GFALM, practically used in \cite{liu2023computing}. Consider the periodic boundary conditions for \eqref{eq:GFALM-semi-disc.}, which can then be efficiently implemented in the spatial direction by the Fourier pseudospectral discretization \cite{shen2011spectral}. 
For notational simplicity, we present the one-dimensional case, i.e., $d=1,\mathbf{x}=x$. Higher dimensional cases can be done in the same manner by tensor-product constructions.

\subsubsection{GFALM with Fourier pseudospectral discretization}

Let the computational domain be $\Omega=[x_{0},\, x_{0}+L]\subset \mathbb{R}$ and set the spatial step size $h:=L/M$ with $M>0$ a positive even integer. Denote $ \varGamma_{h} := \left\{ 0,1,\dots,M-1 \right\}, \ \Omega_{h} := \{x_{j}=x_{0}+jh: j=0,1,\dots,M-1\},$ and $X_{h}:= \{ u_h=(u_{0}, u_{1}, \dots, u_{M-1})^{\top} \colon u_{j}\in \mathbb{R},\, j\in \varGamma_{h} \}$. Let $Y_{M}:={\rm span}\{\varphi_{j}(x) \colon j\in\varGamma_{h}\}$ be the cardinal trigonometric interpolation space where
\begin{align*}
	\varphi_{j}(x)=\frac{1}{M}\sum_{l=-M/2}^{M/2}\frac{1}{a_{l}}\,\mathrm{e}^{\mathrm{i}l\sigma(x-x_{j})} \quad \text{with}\quad
	a_{l}=\begin{cases}
		1, & |l|<M/2,\\
		2, & |l|=M/2,
	\end{cases}
	\quad  \sigma:=\frac{2\pi}{L},
\end{align*}
satisfying the cardinality property $\varphi_{j}(x_{l})=\delta_{j,l}$. The interpolation operator $\mathcal{I}_{h}:L^{2}(\Omega)\to Y_{M}$ is then defined by
\begin{align}
	\mathcal{I}_{h}u(x):=\sum_{j=0}^{M-1}u_{j}\,\varphi_{j}(x), \quad\text{with}\quad u_{j}:=u(x_{j}). \label{eq:interpolation-Ih}
\end{align}
Differentiating \eqref{eq:interpolation-Ih} twice yields the matrix $\mathbf{D}_{xx}\in \mathbb{R}^{M\times M}$ defined by
\begin{align*}
    \partial_{x}^{2} \mathcal{I}_{h}u(x_{j})
	= \sum_{l=0}^{M-1}u_{l} \, \frac{{\rm d}^{2}}{{\rm d}x^{2}}\varphi_{l}(x_{j})
	=:(\mathbf{D}_{xx}u)_{j}, \qquad (\mathbf{D}_{xx})_{j,l}:=\frac{{\rm d}^{2}}{{\rm d}x^{2}}\varphi_{l}(x_{j}).
\end{align*}
First, we introduce some discrete (semi-)norms. For any grid function $u_{h}, v_{h}\in X_{h}$ with periodicity, the discrete inner product, discrete $L^{2}$-norm and forward-difference semi-norm are denoted by
\begin{align*}
	\left\langle u_{h}, v_{h} \right\rangle_{h}:=h\sum_{j=0}^{M-1}u_{j}v_{j}, \quad \|u_{h}\|_{h}:=\sqrt{\langle u_{h},u_{h}\rangle_{h}}, \quad 
    \|\nabla_{h} u_{h}\|_{h}:=\sqrt{\langle \delta_{x}^{+}u_{h},\delta_{x}^{+}u_{h}\rangle_{h}},
\end{align*}
where $(\delta_{x}^{+}u_{h})_{j}:=(u_{j+1}-u_{j})/h$. We also introduce the discrete $H^{1}$-(semi-)norms and discrete $L^{p}$-norms:
\begin{align*}
    &|u_{h}|_{1,h}:=\sqrt{\langle -\mathbf{D}_{xx}u_{h},u_{h}\rangle_{h}}, \quad \|u_{h}\|_{1,h}:=\sqrt{\|u_{h}\|_{h}^{2}+|u_{h}|_{1,h}^{2}}, \quad
	\|u_{h}\|_{h,p}:=\bigg(h\sum_{j=0}^{M-1}|u_{j}|^{p}\bigg)^{1/p}.
\end{align*}

Note that the matrix $-\mathbf{D}_{xx}$ is real symmetric and positive semidefinite \cite{shen2011spectral}. Hence, there exists a unique real symmetric matrix $\big(-\mathbf{D}_{xx}\big)^{\frac{1}{2}}$ such that $-\mathbf{D}_{xx} = \big(-\mathbf{D}_{xx}\big)^{\frac{1}{2}} \big(-\mathbf{D}_{xx}\big)^{\frac{1}{2}}$. Then, for any $u_{h},v_{h}\in X_{h}$, by symmetry and the Cauchy–Schwarz inequality,
\begin{align*}
	\left| \left\langle -\mathbf{D}_{xx}u_{h}, v_{h} \right\rangle_{h} \right|
	\le& \left\| \big(-\mathbf{D}_{xx}\big)^{\frac{1}{2}}u_{h}\right\|_{h} \left\|\big(-\mathbf{D}_{xx}\big)^{\frac{1}{2}} v_{h}\right\|_{h}
	= |u_{h}|_{1,h}\,|v_{h}|_{1,h}.
\end{align*}
Some more basic properties that will be useful are presented below, the proof can be found in, e.g., \cite{coudiere2001discrete,gong2017conservative,kojima2017some}.
\begin{lemma}\label{lem:h1-gap-dx-disc.}
For any grid function $u_h\in X_{h}$ and any $p\ge 1$, the following properties hold.\begin{itemize}
    \item[(i)] Equivalence of the semi-norms:
	$ \|\nabla_{h}u_{h}\|_{h} \le |u_{h}|_{1,h} \le \frac{\pi}{2}\,\|\nabla_{h}u_{h}\|_{h}$.
    \item[(ii)] Discrete Sobolev embedding: $\|u_{h}\|_{h,p+1} \lesssim \sqrt{\|u_{h}\|_{h}^{2} + \|\nabla_{h}u_{h}\|_{h}^{2}}\lesssim \|u_{h}\|_{1,h}$.
    \end{itemize}
\end{lemma}

The GS in the discrete setting reads
\begin{align}
	u_{h}^{*} \in \arg\min \left\{ Q_{h}(u_{h}) \colon u_{h} \in \mathcal{S}_{h, p+1} \right\}, \label{eq:ug-definition-disc.}
\end{align}
where $\mathcal{S}_{h, p+1} := \left\{ u_{h}\in X_{h} : \|u_{h}\|_{h,p+1} = 1 \right\}$ and
\begin{align*}
	Q_{h}(u_{h}) := \frac{1}{2}\left| u_{h} \right|_{1,h}^{2} + h\sum_{j=0}^{M-1}V_{j}\left| u_{j} \right|^{2} + \omega \|u_{h}\|_{h}^{2}.
\end{align*}
with $V_{j}= V(x_{j})$. We now give the following assumption as a discrete analogue of Assumption~\ref{assump:omega-conditions}.

\begin{assumption}\label{assump:omega-conditions-disc.}
	Assume that $\omega$ satisfies $\omega > -\lambda_{0}'$, where
	\begin{align}
		\lambda_{0}' := \inf_{u_{h}\in X_{h}, \|u_{h}\|_{h} = 1}\left( \frac{1}{2}\left| u_{h} \right|_{1,h}^{2} + h\sum_{j=0}^{M-1}V_{j}\left| u_{j} \right|^{2} \right). \label{eq:omega-lambda-0-def.-disc.}
	\end{align}
\end{assumption}

Define the linear operator $\mathcal{A}_{h}:X_h\to X_h$ as
\begin{align*}
	(\mathcal{A}_{h}u_{h})_j := \left( \frac{1}{2}\frac{\delta Q_{h}(u_{h})}{\delta u_{h}} \right)_{j} =  -\frac{1}{2}\left( \mathbf{D}_{xx} u_{h} \right)_{j} + \left( V_{j} + \omega \right)u_{j},\quad \forall\, u_h\in X_h.
\end{align*}
Under Assumption~\ref{assump:omega-conditions-disc.}, $\mathcal{A}_{h}$ is real symmetric and positive definite, and
$Q_h(u_h)=\left\langle \mathcal{A}_h u_h,u_h\right\rangle_h$. 

\begin{lemma}\label{lem:h1-control-by-Q-disc.}
	Under Assumptions~\ref{assump:potential-condition}\&\ref{assump:omega-conditions-disc.}, there exists a constant $C_{0}'>0$ depending only on $\lambda_{0}'$ and $\omega$ such that 
	\begin{align*}
		Q_{h}(v_{h}) \geq C_{0}' \left\| v_{h} \right\|_{1,h}^{2}, \
        \quad \forall v_{h}\in X_{h}.
	\end{align*}
\end{lemma}
\begin{proof}
	The assertion is the discrete analogue of Lemmas \ref{lem:h1-control-by-Q} whose proof follows the continuous case with each estimate in the continuum setting replaced by its discrete counterpart. The details are omitted for brevity. 
\end{proof}

We can now prove the existence of the discrete GS in \eqref{eq:ug-definition-disc.}.

\begin{proposition}[Existence of discrete GS] \label{prop:existence-GS-disc.}
	Under Assumptions~\ref{assump:potential-condition}\&\ref{assump:omega-conditions-disc.}, there exists a GS $u_{h}^{*}$ for the discrete constrained minimization \eqref{eq:ug-definition-disc.}.
\end{proposition}
\begin{proof}
    Let $\{u_{h}^{k} \}_{k=0}^{\infty}\subset \mathcal{S}_{h,p+1}$ be a minimizing sequence with $Q_{h}(u_{h}^{k})\to Q_{g,h}:=\inf_{u_{h} \in \mathcal{S}_{h, p+1}}Q_{h}(u_{h})$ as $k\to\infty$. By Lemma~\ref{lem:h1-control-by-Q-disc.}, $\{u_{h}^{k} \}_{k=0}^{\infty}$ is bounded in $\|\cdot\|_{1,h}$ norm. Since the finite-dimensional vector space $X_h$ is compact (for every fixed $h>0$), $\{u_{h}^{k}\}_{k=0}^{\infty}$ has an accumulation point $u_h^*\in X_h$ with a subsequence $\{u_{h}^{k_j}\}_{j=0}^{\infty}$ satisfying $\|u_{h}^{k_{j}} - u_{h}^{*}\|_{1,h}\rightarrow 0$ as $j\rightarrow \infty$. By the equivalence of norms in the finite-dimensional space, we have
    $\|u_{h}^{k_{j}} - u_{h}^{*}\|_{h,p+1}\rightarrow 0$ (as $j\rightarrow \infty$) and $\|u_{h}^{*}\|_{h,p+1}=\lim_{j\to\infty}\|u_{h}^{k_{j}}\|_{h,p+1}=1$, i.e., $u_{h}^{*}\in\mathcal{S}_{h,p+1}$. Furthermore,
    \begin{align*}
    	\left|Q_h(u_{h}^{k_{j}})-Q_h(u_{h}^{*})\right| = \left|\left\langle \mathcal{A}_h\big(u_{h}^{k_{j}}+ u_{h}^{*}\big),(u_{h}^{k_{j}} - u_{h}^{*})\right\rangle_h \right| \lesssim  \left(\big\|u_{h}^{k_{j}}\big\|_{1,h}+\big\|u_{h}^{*}\big\|_{1,h}\right)\big\|u_{h}^{k_{j}} - u_{h}^{*}\big\|_{1,h}\to0,
    \end{align*}
    yielding that $Q_h(u_{h}^{*})=\lim_{j\to\infty}Q_h(u_{h}^{k_{j}})=Q_{g,h}$. Hence, $u_{h}^{*}\in\mathcal{S}_{h,p+1}$ is a discrete GS of \eqref{eq:ug-definition-disc.}.
\end{proof}

With above discrete ingredients, we can write down the fully discrete GFALM scheme: for $j\in\varGamma_{h}$,
\begin{subequations}\label{eq:GFALM-disc.}
	\begin{align}
		&\frac{1}{\tau}\left( \widetilde{u}_{j}^{n+1} - u_{j}^{n} \right) = -\widetilde{\mu}_{j}^{n+1}, \label{eq:GFALM-suba.} \\
		&\widetilde{\mu}_{j}^{n+1} = -\frac{1}{2}\left( \mathbf{D}_{xx}\widetilde{u}_{h}^{n+1} \right)_{j} + \left( V_{j} + \omega - \widetilde{\lambda}_{h}(u_{h}^{n}) \left| u_{j}^{n} \right|^{p-1} \right)u_{j}^{n} + \alpha\left( \widetilde{u}_{j}^{n+1} - u_{j}^{n} \right), \label{eq:GFALM-subb.} \\
		&u_{j}^{n+1} = \widetilde{u}_{j}^{n+1} / \|\widetilde{u}_{h}^{n+1}\|_{h, p+1}, \label{eq:GFALM-subc.}
	\end{align}
\end{subequations}
where $\alpha\geq0$ is the selected stabilization parameter, and the discrete asymptotic Lagrange multiplier is defined as $\widetilde{\lambda}_{h}(u_{h}) = {Q_{h}(u_{h})}/{\|u_{h}\|_{h,p+1}^{p+1}}=Q_h(u_h)$.

\begin{lemma}[Discrete energy stability]\label{lem:energy-decaying-disc.}
	Under Assumptions~\ref{assump:potential-condition}\&\ref{assump:omega-conditions-disc.}, if
	\begin{align}
		\alpha \ge \frac{1}{2}\max\Big\{0,\ \max_{j\in \varGamma_{h}}\big( V(x_{j})+\omega\big)\Big\} + \frac{1}{2}, \label{eq:alphan-condition-old}
	\end{align}
	the fully discrete GFALM \eqref{eq:GFALM-disc.} possesses the unconditional energy stability for $Q_{h}(u_{h})$: $\forall\,\tau>0$,
	\begin{align}
		Q_{h}(u_{h}^{n+1}) - Q_{h}(u_{h}^{n}) \le -\frac{\tau^{2} \left\| \widetilde{\mu}_{h}^{n+1} \right\|_{1,h}^{2} + 4\tau\|\widetilde{\mu}_{h}^{n+1}\|_{h}^{2}}{2\|\widetilde{u}_{h}^{n+1}\|_{h,p+1}^{2}}. \label{eq:energy-decay-disc.}
	\end{align}
\end{lemma}
\begin{proof}
The proof totally follows the lines of \cite{liu2023computing} and Lemma~\ref{lem:energy_decaying} for the semi-discretization \eqref{eq:GFALM-semi-disc.},  by replacing the continuous inner product and associated functionals with their discrete counterparts. The details are omitted for brevity.
\end{proof}

\subsubsection{Discrete geometry near GS}

To conduct a rigorous local convergence analysis of the fully discrete GFALM scheme \eqref{eq:GFALM-disc.}, we also need to characterize the local geometry structure of the discrete constraint manifold $\mathcal{S}_{h, p+1}$ near $u_{h}^{*}$. We also introduce local coordinates around $u_h^*$, impose a crucial non-degeneracy assumption, and establish an equivalence between the energy error and the $H^1$ error.

Define the constraint functional $G_{h}(u_{h}) := \frac{2}{p+1}\big(\|u_{h}\|_{h,p+1}^{p+1}-1\big)$, where $G_{h}(u_{h})=0$ is equivalent to $u_{h}\in \mathcal S_{h, p+1}$. The associated Lagrangian is
\begin{align}
	L_{h}(u_{h},\lambda) := Q_{h}(u_{h}) - \lambda\,G_{h}(u_{h}), \label{eq:lagrangian-Q-G1-disc.}
\end{align}
and we define the tangent space at $u_{h}\in S_{h, p+1}$ as
\begin{align*}
	\mathcal{T}_{u,h}
	:= \Big\{ \xi_{h} \in X_{h} :\; \mathcal{D}G_{h}(u_{h})[\xi_{h}]
	= 2\left\langle |u_{h}|^{p-1}u_{h}, \xi_{h} \right\rangle_{h} = 0 \Big\}.
\end{align*}
Assume that $\lambda^{*}$ is the corresponding Lagrange multiplier of GS, i.e.,
\begin{align}
	\left( \mathcal{A}_{h}u_{h}^{*} \right)_{j} = \lambda^{*}\, |u^{*}_{j}|^{p-1}u^{*}_{j}.
	\label{eq:euler-lagrange-ug-disc.}
\end{align}
We now introduce a local coordinate  near the GS $u_{h}^{*}$:
\begin{align*}
	\chi_{h}^{}:&\ \mathbb{R}\times \mathcal{T}_{u_{h}^{*}, h}\to X_{h},\quad (r,\xi_{h})\mapsto (1+r)\,u_{h}^{*}+\xi_{h}. 
\end{align*}
Clearly, $\chi_{h}$ is a bijection onto the image, with the inverse 
\begin{align*}
	\chi_{h}^{-1}:\ u_h\in X_{h}\mapsto 
    \bigl(r(u_{h}),\xi_{h}(u_{h})\bigr)
	=\Bigl(\langle |u_{h}^{*}|^{p-1}u_{h}^{*},u_h\rangle_{h}^{} - 1,\,
	u-\langle |u_{h}^{*}|^{p-1}u_{h}^{*},u_h\rangle_{h}^{}\,u_{h}^{*}\Bigr)\in\mathbb{R}\times \mathcal{T}_{u_{h}^{*},h}.
\end{align*}
Then we can establish the following as a discrete analogue of Lemma~\ref{lem:implicit-function-normalization} and Proposition  \ref{prop:coercive-second-variation}.

\begin{lemma}\label{lem:implicit-function-normalization-disc.}
    (i) Under Assumption~\ref{assump:potential-condition}, there exists some $\delta>0$ such that for any $u_{h}\in U_{h,\delta}(u_{h}^{*})= \left\{ u_{h}\in X_{h} \,:\, \|u_{h} - u_{h}^{*}\|_{1,h} \leq \delta \right\}$ admitting the representation $u_{h}=\chi_{h}(r,\xi_{h})$, the implicit relation $\|u\|_{h,p+1}$ $=1$ produces a unique $C^2$ map $ r:\ \xi_{h}\in \mathcal{T}_{u_{h}^{*}, h,\delta}\;\mapsto\;r(\xi_{h})\in\mathbb{R}_{\delta},$ where $\mathcal{T}_{u_{h}^{*},h,\delta} \subset \mathcal{T}_{u_{h}^{*},h}$ and $\mathbb{R}_{\delta}\subset\mathbb{R}$ are neighborhoods of zero, satisfying $r(0)=0$ and
	\begin{align}
		|r(\xi_{h})|\;\lesssim\; \|\xi_{h}\|_{h,p+1}^2 \lesssim \|\xi_{h}\|_{1,h}^{2}. \label{eq:estimate-r}
	\end{align}
    (ii) Under the non-degeneracy condition for some constant $c>0$:
	\begin{align}
		\mathcal{D}_{u}^{2} L_{h}(u_{h}^{*},\lambda^{*})[\xi_{h}, \xi_{h}] \;\ge\; c\,\|\xi_{h}\|_{1,h}^{2},
		\qquad \forall\, \xi_{h} \in \mathcal{T}_{u_{h}^{*},h}, \label{eq:coercivity-second-variation-disc.}
	\end{align}
    there exists a sufficiently small $\delta>0$ such that
	\begin{align}
		Q_{h}(u_{h}) - Q_{h}(u_{h}^{*}) \asymp \|u_{h}-u_{h}^{*}\|_{1,h}^{2}, \quad \forall\, u_{h} \in U_{h,\delta}(u_{h}^{*}) \cap \mathcal{S}_{h,p+1}, \label{eq:energy-gap-quadratic-disc.}
	\end{align}
    and \eqref{eq:ug-definition-disc.} admits $u_{h}^{*}$ as the unique minimizer in $U_{h,\delta}(u_{h}^{*})\cap \mathcal{S}_{h,p+1}$.
\end{lemma}
\begin{proof}
    Statements (i) and (ii) are the discrete counterparts of Lemma~\ref{lem:implicit-function-normalization} and Proposition~\ref{prop:coercive-second-variation}, respectively. In their proofs, the only estimate that does not directly generalize from the continuous case is the discrete version of \eqref{eq:high-order-esti}:
    \begin{align*}
        \left\langle \left( |u_{h}^{*} + \rho\,v_{h} |^{p-1} - |u_{h}^{*}|^{p-1} \right) v_{h}, v_{h} \right\rangle_{h} = o\left( \left\| v_{h} \right\|_{1,h}^{2} \right).
    \end{align*}
    We discuss this in two cases. First, for $1 < p \leq 2$, we have the inequality
    \begin{align*}
        \left| |z_{1}|^{p-1} - |z_{2}|^{p-1} \right| \leq \left| z_{1} - z_{2} \right|^{p-1}
    \end{align*}
    for any $z_{1}, z_{2} \in \mathbb{R}$. Combining this with H\"older's inequality yields
    \begin{align*}
        \left| \left\langle \left( |u_{h}^{*} + \rho\,v_{h} |^{p-1} - |u_{h}^{*}|^{p-1} \right) v_{h}, v_{h} \right\rangle_{h} \right| 
        &\lesssim  \left\| |u_{h}^{*} + \rho\,v_{h} |^{p-1} - |u_{h}^{*}|^{p-1} \right\|_{h,\frac{p+1}{p-1}} \left\| v_{h} \right\|_{h,p+1}^{2} \\
        &\leq  \left( h\sum_{j\in\Gamma_{h}} \left( \left| \rho\,v_{j} \right|^{p-1} \right)^{\frac{p+1}{p-1}} \right)^{\frac{p-1}{p+1}} \left\| v_{h} \right\|_{h,p+1}^{2} \\
        &\leq  \rho^{p-1} \left\| v_{h} \right\|_{h,p+1}^{p-1} \left\| v_{h} \right\|_{h,p+1}^{2} = o\left( \left\| v_{h} \right\|_{1,h}^{2} \right), \quad \forall\, 0\leq \rho \leq 1.
    \end{align*}
    Second, for $p > 2$, we use the Taylor expansion
    \begin{align*}
        |u_{h}^{*} + \rho\,v_{h} |^{p-1} - |u_{h}^{*}|^{p-1} = \rho\,v_{h} \int_{0}^{1}{ (p-1)|u_{h}^{*} + \theta\rho\,v_{h}|^{p-3}\left( u_{h}^{*} + \theta\rho\,v_{h} \right) }\, \mathrm{d}\theta,
    \end{align*}
    which gives
    \begin{align}
        \left| \left\langle \left( |u_{h}^{*} + \rho\,v_{h} |^{p-1} - |u_{h}^{*}|^{p-1} \right) v_{h}, v_{h} \right\rangle_{h} \right| \lesssim \left\| v_{h} \right\|_{h,p+1}^{3} = o\left( \left\| v_{h} \right\|_{1,h}^{2} \right), \quad \forall\, 0\leq \rho \leq 1. \label{eq:high-order-esti-disc.}
    \end{align}
    Apart from this estimate, the proof follows exactly the same lines as in the continuous case, provided all relevant inner products, norms, and functionals are replaced by their discrete counterparts.
\end{proof}

\subsubsection{Exponential convergence rate}
Now we are ready to present the error estimate result for the fully discrete GFALM scheme \eqref{eq:GFALM-disc.}.

\begin{theorem}\label{thm:exponential-convergence-ground-state-disc.}
	Suppose that Assumptions~\ref{assump:potential-condition}\&\ref{assump:omega-conditions-disc.} hold and that the stabilization parameter $\alpha$ satisfies \eqref{eq:alphan-condition-old}. For any fixed $h$, let $u_{h}^{*}$ be the discrete GS defined by \eqref{eq:ug-definition-disc.} and satisfies the non-degeneracy condition in Lemme~\ref{lem:implicit-function-normalization-disc.}. Then, there exists a constant $\delta_{0}>0$ such that for any initial data $u_{h}^{0}\in U_{\delta_{0}}(u_{h}^{*})\cap \mathcal{S}_{p+1}$ and any $\tau > 0$, the fully discrete GFALM scheme \eqref{eq:GFALM-disc.} converges at the rate:
	\begin{align*}
		\|u_{h}^{n}-u_{h}^{*}\|_{1,h} \le C\,\mathrm{e}^{-a  n \tau/(1+\tau)}, \qquad \forall\, n\in \mathbb{N},
	\end{align*}
    where $C,a>0$ are constants independent of $n,\,\tau$.
\end{theorem}

We then turn to the proof of Theorem \ref{thm:exponential-convergence-ground-state-disc.}. Denote
$\mu_{j}^{n} = -\frac{1}{2}\left( \mathbf{D}_{xx}u_{h}^{n} \right)_{j} + ( V_{j} + \omega - \widetilde{\lambda}(u_{h}^{n})\left| u_{j}^{n} \right|^{p-1} )u_{j}^{n}$, then similarly as \eqref{eq:ellipse-equation-mu}, we can obtain the elliptic resolvent equation
\begin{align}
	\left( 1 + \tau\alpha \right) \widetilde{\mu}_{j}^{n+1} - \frac{\tau}{2}\left( \mathbf{D}_{xx}\widetilde{\mu}_{h}^{n+1} \right)_{j} =  \mu_{j}^{n}. \label{eq:ellipse-equation-mu-disc.}
\end{align}
With the discrete $H^{-1}$-norm defined as 
\begin{align*}
	\|u_{h}\|_{-1,h} = \sup_{0\neq v_{h} \in X_{h}} \frac{\left| \left\langle u_{h}, v_{h} \right\rangle_{h} \right|}{\|v_{h}\|_{1,h}}, \quad \forall\, u_{h} \in X_{h},
\end{align*}
we can then derive the following estimates of the numerical solution.

\begin{lemma}\label{lem:uniform-estimates-disc.}
    Under Assumptions \ref{assump:potential-condition}\&\ref{assump:omega-conditions-disc.}, and with $\alpha$ satisfying \eqref{eq:alphan-condition-old}, for the fully discrete GFALM scheme \eqref{eq:GFALM-disc.}, the following  hold:
    \begin{align*}
        &\|u_h^n\|_{1,h}\lesssim 1,\quad 
        \tau\|\widetilde{\mu}_h^{n+1}\|_{1,h}\lesssim 1,\quad
        \|\widetilde{u}_h^{n+1}\|_{1,h}\lesssim 1, \quad
        \|\widetilde{u}_{h}^{n+1} \|_{h,p+1}^{-1}\lesssim 1, 
        \quad \forall n\in\mathbb{N}.
    \end{align*}
\end{lemma}

\begin{lemma}\label{lem:esti-by-energe-gap-disc.}
 Assume that the conditions of Theorem~\ref{thm:exponential-convergence-ground-state-disc.} hold. For the solution of \eqref{eq:GFALM-disc.}, we have
    \begin{align*}
    	\tau\left\| \widetilde{\mu}_{h}^{n+1} \right\|_{1,h} \lesssim  \sqrt{ Q_{h}(u_{h}^{n}) - Q_{h}(u_{h}^{n+1}) }, \quad \left\| u_{h}^{n+1} - u_{h}^{n} \right\|_{1,h} \lesssim \sqrt{ Q_{h}(u_{h}^{n}) - Q_{h}(u_{h}^{n+1}) }.
    \end{align*}
\end{lemma}

\begin{lemma}[Discrete Lyapunov stability]\label{lem:local-stability-ground-state-disc.}
	Let the assumptions in Theorem~\ref{thm:exponential-convergence-ground-state-disc.} hold and let $\delta>0$ be a sufficiently small constant that meets the requirements of Lemma~\ref{lem:implicit-function-normalization-disc.}. Then, there exist $\delta_{0}>0$ such that for all $u_{h}^{0} \in U_{h,\delta_{0}}(u_{h}^{*})\cap \mathcal{S}_{h,p+1}$, the solution of \eqref{eq:GFALM-disc.} with initial guess $u_{h}^{0}$ satisfies
	\begin{align*}
		\|u_{h}^{n}-u_{h}^{*}\|_{1,h} \leq \delta, \qquad \forall\, n\in \mathbb{N}.
	\end{align*}
\end{lemma}

\begin{lemma}[Discrete \L{}ojasiewicz-type gradient inequality]\label{lem:Lojasiewicz-Gradient-Inequality-disc.}
	Under the conditions of Theorem~\ref{thm:exponential-convergence-ground-state-disc.}, there exist some $\delta_{0}>0$ such that for all $u_{h}^{0}\in U_{h,\delta_{0}}(u_{h}^{*})\cap \mathcal{S}_{h,p+1}$,
	\begin{align}
		Q_{h}(u_{h}^{n}) - Q_{h}(u_{h}^{*}) \;\le\; C_{Lo}\,\|\mu_{h}^{n}\|_{-1,h}^{2},
		\qquad \forall\, n\in\mathbb{N}. \label{eq:Lojasiewicz-Gradient-Inequality-disc.}
	\end{align}
	where $0<C_{Lo}<\infty$ is a constant independent of $n,\tau$.
\end{lemma}

Lemmas \ref{lem:uniform-estimates-disc.}--\ref{lem:Lojasiewicz-Gradient-Inequality-disc.} are the discrete analogues of Lemmas \ref{lem:uniform-estimates}, \ref{lem:tilde-u-nplus1-gap-1}, \ref{lem:esti-by-energe-gap}, \ref{lem:local-stability-ground-state}, and \ref{lem:Lojasiewicz-Gradient-Inequality}, respectively. Their proofs follow the same logic as in the continuous case, with the only exception being Lemma \ref{lem:Lojasiewicz-Gradient-Inequality-disc.} which requires the estimate \eqref{eq:high-order-esti-disc.}. The subsequent proof of Theorem \ref{thm:exponential-convergence-ground-state-disc.} essentially follows the same arguments as in the proof of Theorem \ref{thm:exponential-convergence-ground-state}. For the sake of brevity, these proofs are omitted.

\subsection{Numerical Verification}

Now we give some numerical examples to demonstrate the practical convergence behavior of \eqref{eq:GFALM-disc.} with $\alpha$ selected as the minimum value satisfying the condition in Lemma~\ref{lem:energy-decaying-disc.}.

\begin{example}[Soliton in 1D]\label{ex:1}
	Letting $d=1$, $V=0$, $\beta=-1$, and $p=3$ in \eqref{eq:semilinear-elliptic-rot} yields
	\begin{align}
		\frac{1}{2}\partial_{xx}\phi(x) + |\phi(x)|^{2}\phi(x) = \omega \phi(x), \quad x\in \mathbb{R}.  \label{eq:example1}
	\end{align}
    Its GS (up to a translation) is the soliton solution explicitly given as
    \begin{align*}
    	\phi^{*}(x) = \sqrt{2\omega}\,\mathrm{sech}(\sqrt{2\omega}\,x), \quad u^{*}(x) = \phi^{*}(x)/\|\phi^{*}\|_{L^{p+1}}, \quad  x\in \mathbb{R}.
    \end{align*}
    We take $\omega=1$, set the computational domain $\Omega=(-32,32)$, and use a spatial grid of size $M=512$. The initial data for \eqref{eq:GFALM-disc.} is chosen as $u^0_h=\mathcal{I}_{h}\phi_0/\|\phi_0\|_{h,p+1}$ with $\phi_{0}(x) = \mathrm{e}^{-x^{2}/2}$, and the temporal iteration is stopped when the maximal residual of $\widetilde{\mu}^{n+1}$ falls below $10^{-11}$, i.e., $\|\widetilde{\mu}^{n+1}\|_{L^{\infty}} < 10^{-11}.$  
\end{example}

Figure~\ref{fig:ex-loglog-err}(a) displays the evolution of the error $\|u_{h}^{n} - \mathcal{I}_{h}u^{*}\|_{H^{1}}$ with respect to the iteration index $n$ for step sizes $\tau=1,\,0.5,\,0.2,\,0.1$. The curves exhibit a clear exponential decay in $n$.

\begin{figure}[!ht]
	\centering
	\includegraphics[width=0.48\textwidth]{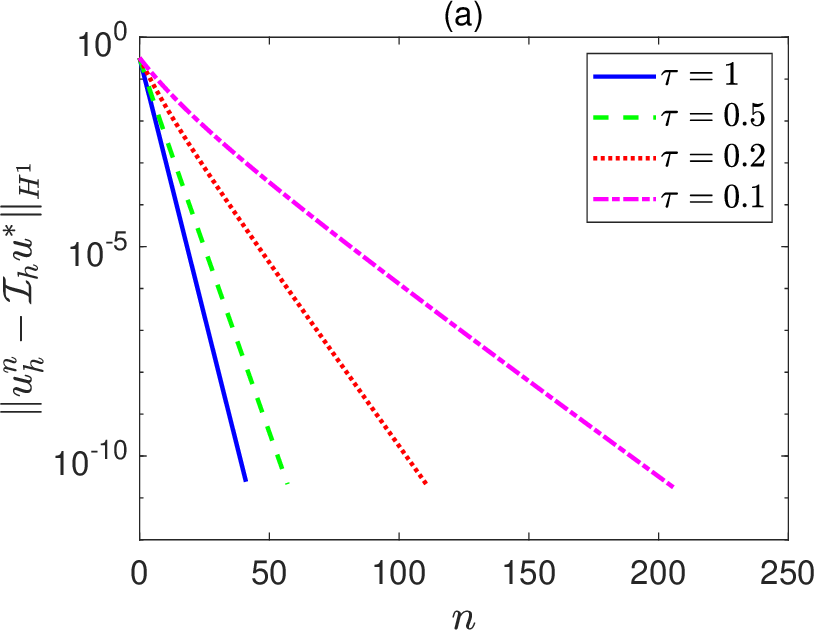} 
	\includegraphics[width=0.48\textwidth]{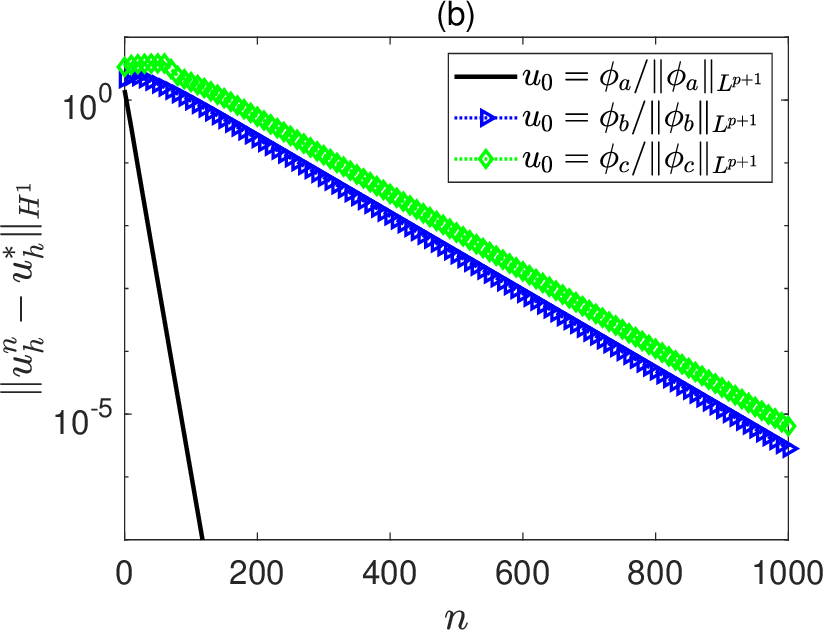}
	\caption{Error of GFALM \eqref{eq:GFALM-disc.} along iteration (on logarithmic scale). (a): $\|u_{h}^{n} - \mathcal{I}_{h}u^{*}\|_{H^{1}}$ under different $\tau$ in Example~\ref{ex:1} (a); (b): $\|u_{h}^{n} - u_{h}^{*}\|_{H^{1}}$ under different initial data $u_{0}$ in Example~\ref{ex:2}.}
	\label{fig:ex-loglog-err}
\end{figure}

\begin{example}[General 2D GS]\label{ex:2}
	Take $d=2$, $\beta=-1$, $p=3$, $\omega=1$, and $V(\mathbf{x})=\left( \gamma_1^{2}x_{1}^{2} + \gamma_2^{2}x_{2}^{2} \right)/2$ with $\gamma_1 = \gamma_2 = 1$ in \eqref{eq:semilinear-elliptic-rot}. The computational domain is taken as $\Omega = [-4,4]\times[-4,4]$ with a spatial grid of size $128\times 128$. The explicit formula of the exact GS is not available in this case and we consider three different initial functions for \eqref{eq:GFALM-disc.}: the $L^{p+1}$-normalizations of 
    \begin{align*}
    	\phi_{a}(x_{1},x_{2}) = \sqrt{\gamma_1\gamma_2/\pi}\,\mathrm{e}^{-(\gamma_1^{2}x_{1}^{2} + \gamma_2^{2}x_{2}^{2})/2},   \quad \phi_{b}(x_{1},x_{2}) = \phi_{a}(x_{1}-1,x_{2}), \quad \phi_{c} = (x_{1}+{\rm i}x_{2})\phi_{a}.
    \end{align*}
    Note that $\phi_{a}$ is real-valued and axisymmetric, while $\phi_{b}$ and $\phi_{c}$ are non-axisymmetric real/complex-valued. Using $\phi_a$, we generate a high-fidelity reference GS $u_{h}^{*}$ by evolving the scheme \eqref{eq:GFALM-disc.} under a small time step $\tau=0.01$ to the final time $t=100$.
\end{example}

Figure~\ref{fig:ex-loglog-err}(b) reports the evolution of the error $\|u_h^{n}-u_h^{*}\|_{H^{1}}$ under $\tau = 0.1$. It is observed that, across all tested initial data (and some more not shown for brevity), the solution of GFALM \eqref{eq:GFALM-disc.} converges towards the same accumulation point at the exponential rate as $n$ increases. The above numerical results verify both the global convergence in Theorem \ref{thm:convergent-subsequences} and the exponential convergence rate of GFALM \eqref{eq:GFALM-disc.}.

\subsection{Discussion and perspective}
We end this section by making some remarks on the presented exponential convergence results, i.e., Theorems \ref{thm:exponential-convergence-ground-state}\&\ref{thm:exponential-convergence-ground-state-disc.}, for the mathematical assumptions and possible generalizations. 
\begin{itemize}
    \item[i)] \emph{Complex-valued case:}  
    The real-valued setup considered in the above analysis can indeed provide the non-degeneracy assumed in Proposition~\ref{prop:coercive-second-variation}. In the complex-valued setting, continuous symmetries typically generate neutral directions (zero modes), so strict coercivity may fail. For instance, the phase invariance $u \mapsto e^{\mathrm{i}\theta}u$ leaves the energy unchanged, making $u_{g}$ a local minimizer only modulo phase. Apart from the non-degeneracy assumption, the core estimates and convergence arguments above remain valid in complex-valued contexts including the rotating case \cite{liu2023computing,liu2025action}. Consequently, provided that Proposition~\ref{prop:coercive-second-variation} holds in the target setting (for example, by imposing symmetry constraints or by establishing coercivity on the orthogonal complement of the zero modes), the exponential convergence result as well as the proof can be extended without essential difficulties.

   \item[ii)] \emph{Global exponential convergence:} 
   The GFALM starting from different regions of the constraint manifold $\mathcal{S}_{p+1}$ may converge to distinct critical points. A central open question is whether, for any given initial guess, the GFALM sequence necessarily converges to a critical point of $Q$ on $\mathcal{S}_{p+1}$. A feasible analytical route is to build on the subsequential convergence results established in Section~\ref{sec.2.2}, which, under mild assumptions, guarantee the existence of at least one convergent subsequence whose limit is a critical point. If, in addition, one can establish a (global) \L{}ojasiewicz-type gradient inequality valid on a relatively large subregion of $\mathcal{S}_{p+1}$, or ensure the critical point possesses similar non-degenerate conditions, then it may be possible to upgrade “subsequence convergence” to convergence of the full sequence, thereby yielding a comprehensive global convergence theory. 

    \item[iii)] \emph{Prior error estimate:} 
    The discrete GS $u^*_h$ within discussion in general differs from the exact GS $u^*$ of \eqref{eq:ug-definition}. As numerically tested on the one-dimensional soliton example with analytical formula \cite{liu2023computing}, their difference exhibits spectral convergence to zero under a vanishing spatial mesh size. To mathematically justify this, one would need an independent prior error estimate between $u^*_h$ and $u^*$, which combining with Theorem~\ref{thm:exponential-convergence-ground-state-disc.} can form the complete error analysis  of GFALM. Such prior estimates exist for the mass constrained energy GS problem, e.g., \cite{cances2010numerical, zhou2003analysis}. However, the direct application of existing frameworks to our problem also  faces fundamental obstacles from the lack of an inner product in $L^{p+1}$. Addressing the issues likely requires refined nonconvex constrained optimization tools which is beyond the scope of this paper and will be addressed in a future work.
\end{itemize}

\section{On original continuous gradient flow}\label{sec:convergence-nGF}

Note that the GFALM scheme was proposed based on the asymptotic flow \eqref{eq:GFALM}, and the original flow \eqref{eq:constrained-gradient-flow} holds independent interest which could give rise to many more possible numerical approximations. 
Thus, in this section, we investigate the continuous gradient flow \eqref{eq:constrained-gradient-flow} and also establish a local exponential convergence theory for it. 
For simplicity, we shall restrict our study to one dimension: $d=1$, and assume that the gradient flow \eqref{eq:constrained-gradient-flow} admits a unique global solution $u(t):=u(\cdot,t)$ with the following regularity assumption:
\begin{align}
	u \in C^{1}\left( (0,\infty); H^1(\Omega) \right). \label{eq:regularity-for-GF} 
\end{align}
Noting that $H^1(\Omega)\hookrightarrow L^{2p}(\Omega)$ and $H^1(\Omega)\hookrightarrow L^{\infty}(\Omega)$ for $d=1$, the Lagrange multiplier $\lambda(u) = {\left\langle \mathcal{A}u, |u|^{p-1} u\right\rangle}\big/{\|u\|_{L^{2p}}^{2p}}$ is well-defined along the flow \eqref{eq:constrained-gradient-flow}. Moreover, by the energy decaying property \eqref{eq:lpplus1-conservation-energy-decay} and Lemma~\ref{lem:h1-control-by-Q}, we have the uniform $H^1$ bound: $\|u(t)\|_{H^1}\lesssim 1$, $\forall\,t>0$.

\begin{lemma}\label{lem:local-stability-ground-state-continous}
	Under Assumptions~\ref{assump:potential-condition}\&\ref{assump:omega-conditions}, \eqref{eq:coercivity-second-variation} and \eqref{eq:regularity-for-GF}, for any sufficiently small $\delta > 0$ that meets the requirements of Lemma~\ref{lem:implicit-function-normalization} and Proposition~\ref{prop:coercive-second-variation}. Then there exists $\delta_{0}$ such that 
    for all $u(0)\in U_{\delta_{0}}(u_{g})\cap\mathcal{S}_{p+1}$, the solution to \eqref{eq:constrained-gradient-flow} with $u(0)$ as initial data satisfies
	\begin{align*}
		\|u(t)-u_{g}\|_{H^{1}} \le \delta, \qquad \forall\, t\ge 0.
	\end{align*}
\end{lemma}
\begin{proof}
    From \eqref{eq:energy-gap-quadratic}, there exist constants $0<c_1\leq c_2$ such that
    \begin{align*}
    	c_1 \left\| u-u_{g} \right\|_{H^{1}} \leq Q(u) - Q(u_{g}) \leq c_2\left\| u - u_{g} \right\|_{H^{1}},\quad\forall\,u\in U_{\delta}(u_{g})\cap\mathcal{S}_{p+1}.
    \end{align*}
    Fix $\delta_{0}$ such that $\delta_{0} < \delta \sqrt{c_1/c_2}$. We now proceed to show that the conclusion of this lemma holds for this $\delta_{0}$. We argue by contradiction. Suppose that there exists $t_{1}>0$ such that $\|u(t_{1})-u_{g}\|_{H^{1}}>\delta$. By the continuity of the mapping $t\mapsto \|u(t)-u_{g}\|_{H^{1}}$, there exists $t_{2}\in (0,t_{1})$ with $\|u(t_{2})-u_{g}\|_{H^{1}}=\delta$. On the one hand, using the lower quadratic bound at $u(t_{2})$,
    \begin{align*}
		Q\big(u(t_{2})\big) \;\ge\; Q(u_{g}) + c_1\left\| u(t_{2}) - u_{g} \right\|_{H^{1}} \;\ge\; Q(u_{g}) + c_1\delta^{2}.
	\end{align*}
    On the other hand, by the choosing of $\delta_{0}$, we have $c_2\delta_{0}^{2}<c_1\delta^{2}$, and therefore,
	\begin{align*}
		Q\big(u(0)\big) \le Q(u_{g}) + c_2\|u(0)-u_{g}\|_{H^{1}}^{2}
		\le Q(u_{g}) + c_2\delta_{0}^{2}< Q\big(u(t_{2})\big).
	\end{align*}
    This contradicts with \eqref{eq:lpplus1-conservation-energy-decay}.
\end{proof}

For $u\in U_{\delta}(u_{g})\cap\mathcal{S}_{p+1}$, denote $v = u - u_{g}$ and  $\mathcal{N}(u,u_{g}) := \lambda(u)|u|^{p-1}u - \lambda_{g}|u_{g}|^{p-1}u_{g}$. Then
\begin{align}
	\mathcal{N}(u,u_{g}) =&\, \lambda_{g} \big( |u|^{p-1}u - |u_{g}|^{p-1}u_{g} \big) + \big( \lambda(u) - \lambda_{g} \big) |u_{g}|^{p-1}u_{g} \nonumber\\
	&\, + \big( \lambda(u) - \lambda_{g} \big)\big( |u|^{p-1}u - |u_{g}|^{p-1}u_{g} \big). \label{eq:nonlinear-term-N1}
\end{align}
Now we estimate $\lambda(u)-\lambda_g$. Writing $\lambda(u) := f_{1}(u) / g_{1}(u)$ with $f_{1}(u) := \langle \mathcal{A}u, |u|^{p-1}u \rangle$ and $g_{1}(u) := \|u\|_{L^{2p}}^{2p}$, one has
\begin{align*}
	\mathcal{D}\lambda(u)[v] = \frac{g_{1}(u)\, \mathcal{D}f_{1}(u)[v] - f_{1}(u)\, \mathcal{D}g_{1}(u)[v]}{(g_{1}(u))^{2}}, 
\end{align*}
with
$\mathcal{D}f_{1}(u)[v] = \left\langle \mathcal{A}u, p\,|u|^{p-1}v \right\rangle + \left\langle \mathcal{A}v, |u|^{p-1}u \right\rangle, \  \mathcal{D}g_{1}(u)[v] = 2p\, \langle |u|^{2p-2}u, v \rangle$. By H\"older’s inequality and the Sobolev embedding, we obtain
\begin{align*}
	| \mathcal{D}f_{1}(u)[v] | \lesssim&\, \|u\|_{L^{\infty}}^{p-1}\|u\|_{H^{1}} \|v\|_{H^{1}} \lesssim \|u\|_{H^{1}}^{p}\|v\|_{H^{1}} \lesssim \|v\|_{H^{1}}, \\
	 | \mathcal{D}g_{1}(u)[v] | \lesssim&\, \|u\|_{L^{\infty}}^{p-1}\|u\|_{L^{p+1}}^{p-1} \|v\|_{L^{p+1}} \lesssim \|u\|_{H^{1}}^{2p-1}\|v\|_{H^{1}} \lesssim \|v\|_{H^{1}}.
\end{align*}
For $g_{1}$, we have the Taylor expansion
\begin{align*}
	g_{1}(u) - g_{1}(u_{g}) = \int_{0}^{1} \mathcal{D}g_{1}(u_{g}+\rho v)[v]\,\mathrm{d}\rho = \mathcal{O}(\|v\|_{H^{1}}).
\end{align*}
Since $ |g_{1}(u_{g})| > 0 $, for sufficiently small $\delta$, $g_{1}(u)$ remains bounded away from zero, which yields $|\mathcal{D}\lambda(u)[v] | \lesssim \| v \|_{H^{1}}$, and consequently the Taylor expansion,
\begin{align}
	\lambda(u) - \lambda_g = \lambda(u) - \lambda(u_{g}) = \int_{0}^{1} \mathcal{D}\lambda(u_{g}+\rho v)[v]\,\mathrm{d}\rho = \mathcal{O}(\|v\|_{H^{1}}). \label{eq:lambda-u-lambda-ug}
\end{align}
Note that, by writing $v=u-u_g=r(\xi)u_g+\xi$ with $\xi\in\mathcal{T}_{u_g}$, we have $\left\langle |u_{g}|^{p-1}u_{g}, v \right\rangle=r(\xi)=\mathcal{O}(\|\xi\|_{H^1}^2)=\mathcal{O}(\|v\|_{H^1}^2)$. 
Substituting \eqref{eq:lambda-u-lambda-ug} into \eqref{eq:nonlinear-term-N1} and using process similar to \eqref{eq:high-order-esti-2} yields
\begin{align}
    \left\langle \mathcal{N}(u,u_{g}), v \right\rangle
     =&\, p\,\lambda_{g}\left\langle |u_{g}|^{p-1}v, v \right\rangle + \big( \lambda(u) - \lambda_{g} \big) \left\langle |u_{g}|^{p-1}u_{g}, v \right\rangle + o\left( \|v\|_{H^{1}}^{2} \right) \nonumber \\
     =&\, p\,\lambda_{g}\left\langle |u_{g}|^{p-1}v, v \right\rangle + o\left( \|v\|_{H^{1}}^{2} \right). \label{eq:Nu-ug}
\end{align}

Denote the manifold gradient of $Q(u)$ as $\mathcal{F}(u) := \mathcal{A} u - \lambda(u)|u|^{p-1}u$, and the \L{}ojasiewicz-type gradient inequality near  $u_{g}$ reads as follows.

\begin{lemma}\label{lem:Lojasiewicz-Gradient-Inequalityug}
	Under Assumptions~\ref{assump:potential-condition}\&\ref{assump:omega-conditions}, \eqref{eq:coercivity-second-variation} and \eqref{eq:regularity-for-GF}, there exists some $\delta_{0}>0$ such that for $u(0)\in U_{\delta_{0}}(u_{g})\cap\mathcal{S}_{p+1}$, the solution to \eqref{eq:constrained-gradient-flow} satisfies
	\begin{align}
		Q(u(t)) - Q(u_{g}) \lesssim\|\mathcal{F}(u(t))\|_{H^{-1}}^{2}
		\qquad \forall\, t\geq 0. \label{eq:Lojasiewicz-Gradient-Inequalityug}
	\end{align}
\end{lemma}
\begin{proof}
    By Lemma~\ref{lem:local-stability-ground-state-continous} and \eqref{eq:Nu-ug}, for $v(t):=u(t)-u_g$, we have $\|v(t)\|_{H^1}\leq\delta$ and
    \begin{align*}
		\langle \mathcal{F}(u(t)), v(t) \rangle
        &=\langle \mathcal{A}u(t), v(t) \rangle - \lambda(u(t)) \langle |u(t)|^{p-1}u(t), v(t) \rangle \\
		&=\langle \mathcal{A}u(t), v(t) \rangle - \lambda_g\langle |u_{g}|^{p-1}u_g, v(t) \rangle -  \langle \mathcal{N}(u(t),u_{g}), v(t) \rangle \\
        &=\langle \mathcal{A}u(t), v(t) \rangle - \langle \mathcal{A}u_{g}, v(t) \rangle - p\,\lambda_{g} \left\langle |u_{g}|^{p-1}v(t), v(t) \right\rangle + o\!\left( \|v(t)\|_{H^{1}}^{2} \right) \\
        &= \langle \mathcal{L}v(t), v(t) \rangle + o\!\left( \|v(t)\|_{H^{1}}^{2} \right).
	\end{align*}
	Noting that $\langle\mathcal{L}v(t),v(t)\rangle \asymp \|v(t)\|_{H^{1}}^{2}$ implied by the non-degeneracy condition \eqref{eq:coercivity-second-variation}, we find that when $\delta$ is small enough,
    \begin{align*}
		\|\mathcal{F}(u(t))\|_{H^{-1}} 
        \ge \frac{ \langle \mathcal{F}(u(t)), v(t) \rangle }{\|v(t)\|_{H^{1}}}
        = \frac{\langle \mathcal{L}v(t), v(t) \rangle}{\|v(t)\|_{H^{1}}} + o\!\left( \|v(t)\|_{H^{1}} \right)
        \gtrsim \|v(t)\|_{H^{1}},\quad \forall\, t\geq0.
	\end{align*}
    By Proposition~\ref{prop:coercive-second-variation} and Lemma~\ref{lem:local-stability-ground-state-continous}, the local equivalence between $\|v(t)\|_{H^{1}}^{2}$ and $Q(u(t))-Q(u_{g})$ holds for all $t\geq0$, which completes the proof.
\end{proof}

Now we are ready to state and prove the exponential convergence rate on the continuous flow \eqref{eq:constrained-gradient-flow}. 

\begin{theorem}[Exponential convergence of original flow]\label{thm:exponential-convergence-ground-state-continuous}
	Under Assumptions~\ref{assump:potential-condition}\&\ref{assump:omega-conditions}, \eqref{eq:coercivity-second-variation} and \eqref{eq:regularity-for-GF}. Then, there exists some  $\delta_{0}>0$ such that the flow \eqref{eq:constrained-gradient-flow} with $u(0)\in U_{\delta_{0}}(u_{g})\cap\mathcal{S}_{p+1}$ satisfies
	\begin{align*}
		\|u(t)-u_{g}\|_{H^{1}} \le C\,{\rm e}^{-a t}, \qquad \forall\, t>0,
	\end{align*}
	where $C,a>0$ are constants independent of $t$.
\end{theorem}
\begin{proof}
    The decaying property \eqref{eq:lpplus1-conservation-energy-decay}  states that
    \begin{align*}
    	\frac{{\rm d}}{{\rm d}t} Q(u(t)) = -2\|\mathcal{F}(u(t))\|_{L^{2}}^{2} \le -c_{D}\,\|\mathcal{F}(u(t))\|_{H^{-1}}^{2},
    \end{align*}
    for some $c_{D}>0$. Combining with the \L{}ojasiewicz-type gradient inequality \eqref{eq:Lojasiewicz-Gradient-Inequalityug} yields 
    \begin{align*}
    	\frac{{\rm d}}{{\rm d}t}\big[Q(u(t)) - Q(u_{g})\big] \le -\frac{c_{D}}{C_L}\,\big[Q(u(t)) - Q(u_{g})\big],
    \end{align*}
    with some $C_L>0$. 
    By Gr\"onwall’s inequality and setting $a_{0}:=\frac{c_{D}}{C_{L}}$, we obtain
    \begin{align*}
    	Q(u(t)) - Q(u_{g}) \le {\rm e}^{-a_{0} t}\,\big[Q(u(0)) - Q(u_{g})\big].
    \end{align*}
    The local equivalence $Q(u(t))-Q(u_{g}) \asymp \|u(t)-u_{g}\|_{H^{1}}^{2}$  finishes the proof with $a=a_{0}/2$.
\end{proof}

\section{Conclusion} \label{sec. 5}
The action ground state of the focusing nonlinear Schr\"odinger equation can be formulated as an $L^{p+1}$-constrained minimization ($p>1$) \cite{liu2023computing}, generalizing the traditional $L^2$-constrained energy ground state problem. To compute these action ground states, a Gradient Flow with Asymptotic Lagrange Multiplier (GFALM) method has been proposed in \cite{liu2023computing}. While GFALM's practical effectiveness has been demonstrated, rigorous analysis has been challenging due to the $L^{p +1}$ constraint, which lacks the inner-product structure present in the energy ground state problem. In this work, we carried out systematic convergence analysis for GFALM.  Specifically, we proved a general global convergence result, guaranteeing the existence of accumulation points and convergent subsequences for the semi-discrete GFALM scheme. Furthermore, we developed a novel local convergence framework tailored to the $L^{p+1}$-spherical manifold, demonstrating the numerically observed exponential convergence rate of the GFALM scheme under suitable non-degeneracy assumptions. This involved characterizing the local geometry of the constraint manifold, establishing a quadratic growth property of the functional, and deriving a \L{}ojasiewicz-type gradient inequality. A parallel exponential convergence analysis was also provided for the continuous normalized gradient flow, laying the groundwork for future numerical discretization designs.

\section*{Acknowledgement}
W.~Liu is supportted by the NSFC grant 12571448 and the Innovation Research Foundation of NUDT (202402-YJRC-XX-002). T. Wang and X. Zhao are supported by National Key Research and Development Program of China, National MCF Energy R\&D Program (No. 2024YFE03240400), NSFC 42450275, 12271413.

\appendix
\section{Proof of Lemma~\ref{lem:implicit-function-normalization}}\label{appdx:proof-lem:implicit-function-normalization}
	Rewrite the relation $\|u\|_{L^{p+1}}=1$ as $ G(u) \equiv 0 $ and $ u := \chi(r, \xi) $. Defining an auxiliary mapping $ \widetilde{G}(r, \xi) := G\bigl(\chi(r, \xi)\bigr) $, we compute the (G\^ateaux) derivatives of $ \widetilde{G} $ with respect to each variable:
	\begin{align*}
		\partial_{r}\widetilde{G}(r, \xi) &= 2\,\left\langle |\chi(r, \xi)|^{p-1}\chi(r, \xi), u_{g} \right\rangle, \\
		\mathcal{D}_{\xi}\widetilde{G}(r, \xi)[v] &= 2\,\left\langle |\chi(r, \xi)|^{p-1}\chi(r, \xi), v \right\rangle, \quad \forall v\in \mathcal{T}_{u_{g}}.
	\end{align*}
	With $p>1$, one finds
	\begin{align*}
		\widetilde{G} \in C^{2}(\mathbb{R} \times \mathcal{T}_{u_{g}}; \mathbb{R}), \quad \widetilde{G}(0, 0) = 0, \quad \partial_{r}\widetilde{G}(0, 0) = 2 \neq 0.
	\end{align*}
	By the implicit function theorem, we obtain the first assertion of Lemma~\ref{lem:implicit-function-normalization}.

    Clearly, $r(0)=0$. To compute the G\^{a}teaux derivatives of $ r(\xi) $, we set $ \widehat{G}(\xi) := \widetilde{G}(r(\xi), \xi) \equiv 0$. By the chain rule,
	\begin{align*}
		\mathcal{D}\widehat{G}(\xi)[v] = \partial_{r}\widetilde{G}(r(\xi), \xi)\, \mathcal{D}r(\xi)[v] + \mathcal{D}_{\xi}\widetilde{G}(r(\xi), \xi)[v] = 0, \quad \forall v\in \mathcal{T}_{u_{g}}.
	\end{align*}
	Hence, 
	\begin{align*}
		\mathcal{D}r(\xi)[v] = -\frac{\mathcal{D}_{\xi} \widetilde{G}(r(\xi), \xi)[v]}{\partial_{r}\widetilde{G}(r(\xi), \xi)} = -\frac{\left\langle |\chi(r(\xi), \xi)|^{p-1} \chi(r(\xi), \xi), v \right\rangle}{\left\langle |\chi(r(\xi), \xi)|^{p-1}\chi(r(\xi), \xi), u_{g} \right\rangle}.
	\end{align*}
	Note that $ \mathcal{D}r(0)[v] = 0 $ for all $v\in \mathcal{T}_{u_{g}}$. For any $v_{1}\in \mathcal{T}_{u_{g}}$, define
	\begin{align*}
		F(r, \xi) := -\frac{f_{2}(r, \xi)}{g_{2}(r, \xi)} := -\frac{\left\langle |\chi(r, \xi)|^{p-1}\chi(r, \xi), v_{1} \right\rangle}{\left\langle |\chi(r, \xi)|^{p-1}\chi(r, \xi), u_{g} \right\rangle}.
	\end{align*}
	Let $ \widetilde{F}(\xi) = F(r(\xi), \xi) $. Then,
	\begin{align}
		\mathcal{D}^{2}r(\xi)[v_{2}, v_{1}] = \mathcal{D}\widetilde{F}(\xi)[v_{2}]  = \partial_{r}F(r(\xi), \xi)\, \mathcal{D}r(\xi)[v_{2}] + \mathcal{D}_{\xi}F(r(\xi), \xi)[v_{2}], \quad \forall v_{1}, v_{2}\in \mathcal{T}_{u_{g}}, \label{eq:implicit-function-d2r discrete}
	\end{align}
	and we handle each term on the right-hand-side of \eqref{eq:implicit-function-d2r discrete} individually. Firstly, we have
	\begin{align*}
		\partial_{r}F(r,\xi) = - \frac{g_{2}(r,\xi)\,\partial_{r}f_{2}(r,\xi) - f_{2}(r,\xi)\, \partial_{r} g_{2}(r,\xi)}{(g_{2}(r,\xi))^2},
	\end{align*}
	with
	$\partial_{r}f_{2}(r,\xi) = p\,\left\langle |\chi(r,\xi)|^{p-1}u_{g}, v_{1} \right\rangle, \ \partial_{r}g_{2}(r,\xi) = p\,\left\langle |\chi(r,\xi)|^{p-1}u_{g}, u_{g} \right\rangle.$
	Secondly, 
	\begin{align*}
		\mathcal{D}_{\xi}F(r,\xi)[v_{2}] = - \frac{g_{2}(r,\xi)\,\mathcal{D}_{\xi}f_{2}(r,\xi)[v_{2}] - f_{2}(r,\xi)\,\mathcal{D}_{\xi}g_{2}(r,\xi)[v_{2}]}{(g_{2}(r,\xi))^2},
	\end{align*}
	with
	$\mathcal{D}_{\xi}f_{2}(r,\xi)[v_{2}] = p\,\left\langle |\chi(r,\xi)|^{p-1}v_{2}, v_{1} \right\rangle$ and $\mathcal{D}_{\xi}g_{2}(r,\xi)[v_{2}] = p\,\left\langle |\chi(r,\xi)|^{p-1}v_{2}, u_{g} \right\rangle$. 
	By the H\"older inequality and noting that $\|\chi(r(\xi), \xi)\|_{L^{p+1}}=\|u_{g}\|_{L^{p+1}}=1$, we have
    \begin{align*}
		&|f_{2}(r(\xi), \xi)| \leq \|v_{1}\|_{L^{p+1}}, 
        \quad |g_{2}(r(\xi), \xi)| \leq 1, \\
        &|\partial_{r}f_{2}(r(\xi), \xi)|\leq p\|v_{1}\|_{L^{p+1}}, 
        \quad |\partial_{r}g_{2}(r(\xi), \xi)|\leq p, \\
		&\big|\mathcal{D}_{\xi}f_{2}(r(\xi), \xi)[v_{2}]\big| \leq p\|v_{1}\|_{L^{p+1}}\|v_{2}\|_{L^{p+1}}, 
        \quad \big|\mathcal{D}_{\xi}g_{2}(r(\xi), \xi)[v_{2}]\big| \leq p\|v_{2}\|_{L^{p+1}}.
	\end{align*}
	Denote $\widehat{g}_{2}(u) = \left\langle |u|^{p-1}u, u_{g} \right\rangle$. Since $ \widehat{g}_{2}(u_{g}) = \|u_{g}\|_{L^{p+1}}^{p+1} = 1 $, for $ u \in U_{h,\delta}(u_{g})\cap \mathcal{S}_{h,p+1} $ we have
	\begin{align*}
		|\widehat{g}_{2}(u)-\widehat{g}_{2}(u_{g})| \leq&\,  \int_{\Omega}{ \left| |u|^{p-1}u - |u_{g}|^{p-1}u_{g} \right| \left| u_{g} \right| }\,\text{d}\mathbf{x} \\
		\leq&\, p\int_{\Omega}{ \left( |u|^{p-1} + |u_{g}|^{p-1} \right) \left| u-u_{g} \right| \left| u_{g} \right| }\,\text{d}\mathbf{x} \\ 
		\leq&\, p\left( \|u\|_{L^{p+1}}^{p-1} + \|u_{g}\|_{L^{p+1}}^{p-1} \right)\|u-u_{g}\|_{L^{p+1}} \|u_{g}\|_{L^{p+1}} \lesssim \delta. 
	\end{align*}
	Thus, when $\delta$ is sufficiently small, $ |g_{2}(r(\xi), \xi)| $ is bounded away from zero. Plugging the above bounds into \eqref{eq:implicit-function-d2r discrete} yields
	\begin{align} 
		\Big|\mathcal{D}^{2}r(\xi)[v_{2},v_{1}]\Big| 
        \leq \big|\partial_{r}F(r(\xi),\xi)\big|\,\big|\mathcal{D}r(\xi)[v_{2}]\big| + \big|\mathcal{D}_{\xi}F(r(\xi),\xi)[v_{2}]\big| \lesssim \|v_{1}\|_{L^{p+1}} \|v_{2}\|_{L^{p+1}}. 
        \label{eq:implicit-function-d2r-estimate-disc.}	
	\end{align}
	According to the Taylor expansion with integral remainder,
	\begin{align*}
		r(\xi) = r(0) + \mathcal{D}r(0)[\xi] + \int_{0}^{1}(1-\rho)\mathcal{D}^{2}r(\rho\xi)[\xi,\xi] \, \mathrm{d}\rho = \int_{0}^{1}(1-\rho)\mathcal{D}^{2}r(\rho\xi)[\xi,\xi] \, \mathrm{d}\rho.
	\end{align*} 
	Using \eqref{eq:implicit-function-d2r-estimate-disc.} gives $ |r(\xi)| \lesssim \|\xi\|_{L^{p+1}}^{2} $, which confirms the second assertion of Lemma~\ref{lem:implicit-function-normalization}.


\begin{thebibliography}{10}

\bibitem{altmann2021jmethod}
{\sc R.~Altmann, P.~Henning, and D.~Peterseim}, {\em The {J}-method for the
  {G}ross--{P}itaevskii eigenvalue problem}, Numer. Math., 148 (2021),
  pp.~575--610.

\bibitem{antoine2017efficient}
{\sc X.~Antoine, A.~Levitt, and Q.~Tang}, {\em Efficient spectral computation
  of the stationary states of rotating {B}ose--{E}instein condensates by
  preconditioned nonlinear conjugate gradient methods}, J. Comput. Phys.,
  343 (2017), pp.~92--109.

\bibitem{ardila2021global}
{\sc A.~H. Ardila and H.~Hajaiej}, {\em Global well-posedness, blow-up and
  stability of standing waves for supercritical {NLS} with rotation}, J. Dyn. Differ. Equ.,
  35 (2023), pp.~1643--1665.

\bibitem{bao2013mathematical}
{\sc W.~Bao and Y.~Cai}, {\em Mathematical theory and numerical methods for
  {Bose--Einstein} condensation}, Kinet. Relat. Models, 6 (2013),
  pp.~1--135.

\bibitem{bao2004computing}
{\sc W.~Bao and Q.~Du}, {\em Computing the ground state solution of
  {Bose--Einstein} condensates by a normalized gradient flow}, SIAM J. Sci. Comput., 25 (2004), pp.~1674--1697.

\bibitem{berestycki1983nonlinearI}
{\sc H.~Berestycki and P.-L. Lions}, {\em Nonlinear scalar field equations,
  {I}. existence of a ground state}, Arch. Ration. Mech. Anal.,
  82 (1983), pp.~313--345.

\bibitem{berestycki1983nonlinearII}
{\sc H.~Berestycki and P.-L. Lions}, {\em Nonlinear scalar
  field equations, {II}. existence of infinitely many solutions}, Arch. Ration. Mech. Anal.,
  82 (1983), pp.~347--375.

\bibitem{cances2010numerical}
{\sc E.~Canc{\'e}s, R.~Chakir, and Y.~Maday}, {\em Numerical analysis of
  nonlinear eigenvalue problems}, J. Sci. Comput., 45 (2010),
  pp.~90--117.

\bibitem{chang2007computing}
{\sc S.-H. Chang, C.-C. Chien, and B.-F. Jeng}, {\em Computing wave functions
  of nonlinear {Schr\"odinger} equations: A time-independent approach}, J. Comput. Phys.,
  226 (2007), pp.~104--130.

\bibitem{ChangWenZhao}
{\sc Z. Chang, Z. Wen, X. Zhao}, {\em Deep neural network approaches for computing the defocusing action ground state of nonlinear Schr\"odinger equation}, 
Ann. Appl. Math., 41 (2025), pp.~42–76.  

\bibitem{chen2025second}
{\sc H.~Chen, G.~Dong, J.~A. Iglesias, W.~Liu, and Z.~Xie}, {\em Second-order
  flows for approaching stationary points of a class of nonconvex energies via
  convex-splitting schemes}, SIAM J. Sci. Comput., 47 (2025),
  pp.~A1604--A1627.

\bibitem{CDLX2023JCP}
{\sc H.~Chen, G.~Dong, W.~Liu, and Z.~Xie}, {\em Second-order flows for computing the ground states of rotating Bose-Einstein condensates}, J. Comput. Phys., 475 (2023), p.~111872.

\bibitem{CLLZ2024SINUM}
{\sc Z.~Chen, J.~Lu, Y.~Lu, and X.~Zhang}, {\em On the convergence of Sobolev gradient flow for the Gross-Pitaevskii eigenvalue problem}, SIAM J. Numer. Anal., 62 (2024), pp.667--691.


\bibitem{coudiere2001discrete}
{\sc Y.~Coudi{\'e}re, T.~Gallou{\"e}t, and R.~Herbin}, {\em Discrete {Sobolev}
  inequalities and {$L^p$} error estimates for finite volume solutions of
  convection diffusion equations}, ESAIM Math. Model. Numer. Anal.,
  35 (2001), pp.~767--778.



\bibitem{danaila2010new}
{\sc I.~Danaila and P.~Kazemi}, {\em A new {Sobolev} gradient method for direct
  minimization of the {Gross--Pitaevskii} energy with rotation}, SIAM J. Sci. Comput.,
  32 (2010), pp.~2447--2467.

\bibitem{danaila2017computation}
{\sc I.~Danaila and B.~Protas}, {\em Computation of ground states of the
  {Gross--Pitaevskii} functional via {Riemannian} optimization}, SIAM J. Sci. Comput.,
  39 (2017), pp.~B1102--B1129.


\bibitem{dion2007ground}
{\sc C.~M. Dion and E.~Canc{\'e}s}, {\em Ground state of the time-independent
  {Gross--Pitaevskii} equation}, Comput. Phys. Commun., 177 (2007),
  pp.~787--798.

\bibitem{Dovetta}
{\sc  S. Dovetta, E. Serra, and P. Tilli}, {\em Action versus energy ground states in nonlinear
Schr\"odinger equations}, Math. Ann., 385 (2023) pp.~1545–1576.

\bibitem{faou2018convergence}
{\sc E.~Faou and T.~J{\'e}z{\'e}quel}, {\em Convergence of a normalized
  gradient algorithm for computing ground states}, IMA J. Numer. Anal.,
  38 (2018), pp.~360--376.

\bibitem{feng2025ima}
{\sc Z. Feng, Q. Tang, and C. Wang}, {\em On the discrete normalized gradient flow for computing ground states of rotating Bose–Einstein condensates: energy dissipation and global convergence}, IMA J. Numer. Anal., (2025), https://doi.org/10.1093/imanum/draf100


\bibitem{fukuizumi2001stability}
{\sc R.~Fukuizumi}, {\em Stability and instability of standing waves for the
  nonlinear {Schr\"odinger} equation with harmonic potential}, Discrete Contin. Dyn. Syst.,
  7 (2001), pp.~525--544.

\bibitem{fukuizumi2003instability}
{\sc R.~Fukuizumi and M.~Ohta}, {\em Instability of standing waves for
  nonlinear {Schr\"odinger} equations with potentials}, Differ. Integral Equ.,
  16 (2003), pp.~691--706.

\bibitem{GT}
{\sc D. Gilbarg and N. S. Trudinger}, {\em Elliptic Partial Differential Equations of Second Order}, Springer, 2001.

\bibitem{gong2017conservative}
{\sc Y.~Gong, Q.~Wang, Y.~Wang, and J.~Cai}, {\em A conservative {Fourier}
  pseudo-spectral method for the nonlinear {Schr{\"o}dinger} equation}, J. Comput. Phys., 328 (2017), pp.~354--370.

\bibitem{henning2023}
{\sc P.~Henning}, {\em The dependency of spectral gaps on the convergence of the inverse iteration for a nonlinear eigenvector problem}. Math. Mod. Meth. Appl. S., 33 (2023), pp.~1517--1544

\bibitem{henning2020sobolev}
{\sc P.~Henning and D.~Peterseim}, {\em {Sobolev} gradient flow for the
  {Gross--Pitaevskii} eigenvalue problem: Global convergence and computational
  efficiency}, SIAM J. Numer. Anal., 58 (2020), pp.~1744--1772.

\bibitem{Jeanjean}
{\sc L. Jeanjean and S. Lu}, {\em On global minimizers for a mass constrained problem}, Calc.
Var. Part. Differ. Equ., 64 (2022) p.~214.

\bibitem{kojima2017some}
{\sc H.~Kojima, T.~Matsuo, and D.~Furihata}, {\em Some discrete inequalities
  for central-difference type operators}, Math. Comp.,
  86 (2017), pp.~1719--1739.

\bibitem{liu2021normalized}
{\sc W.~Liu and Y.~Cai}, {\em Normalized gradient flow with {Lagrange}
  multiplier for computing ground states of {Bose--Einstein} condensates}, SIAM J. Sci. Comput.,
  43 (2021), pp.~B219--B242.

\bibitem{liu2025action}
{\sc W.~Liu, C.~Wang, and X.~Zhao}, {\em On action ground states of defocusing
  nonlinear {Schr{\"o}dinger} equations}, Math. Models Methods Appl. Sci., 35 (2025), pp.~39--74.

\bibitem{liu2025computing}
{\sc W.~Liu, Z.~Wen, Y.~Yuan, and X.~Zhao}, {\em Computing defocusing
  action ground state of rotating nonlinear {Schr\"odinger} equation: methods
  via various formulations and comparison}, J. Comput. Phys.,
  538 (2025), p.~114193.


\bibitem{liu2023computing}
{\sc W.~Liu, Y.~Yuan, and X.~Zhao}, {\em Computing the action ground state for
  the rotating nonlinear {Schr{\"o}dinger} equation}, SIAM J. Sci. Comput.,
  45 (2023), pp.~A397--A426.


\bibitem{pohozaev1965eigenfunctions}
{\sc S.~I. Pohozaev}, {\em Eigenfunctions of the equation {$\Delta u + \lambda
  f(u) = 0$}}, Sov. Math. Dokl., 5 (1965), pp.~1408--1411.


\bibitem{shatah1985instability}
{\sc J.~Shatah and W.~A. Strauss}, {\em Instability of nonlinear bound states},
  Comm. Math. Phys., 100 (1985), pp.~173--190.

\bibitem{shen2011spectral}
{\sc J.~Shen, T.~Tang, and L.-L. Wang}, {\em Spectral Methods: Algorithms,
  Analysis and Applications}, vol.~41 of Springer Series in Computational
  Mathematics, Springer-Verlag, Berlin, Heidelberg, 2011.

\bibitem{strauss1977existence}
{\sc W.~A. Strauss}, {\em Existence of solitary waves in higher dimensions},
  Comm. Math. Phys., 55 (1977), pp.~149--162.

\bibitem{WangChunshan}
{\sc C. Wang}, {\em Computing the least action ground state of the nonlinear Schr\"odinger equation by a normalized gradient flow}, J. Comput. Phys., 471 (2022) p.~111675.

\bibitem{wang2014projection}
{\sc H.~Wang}, {\em A projection gradient method for computing ground state of
  spin-2 {Bose-Einstein} condensates}, J. Comput. Phys., 274
  (2014), pp.~473--488.

\bibitem{willem1996minimax}
{\sc M. Willem}, {\em Minimax Theorems}, vol.~24 of Progress in Nonlinear Differential Equations and Their Applications, Birkh\"auser Boston, Boston, MA, 1996.


\bibitem{zhang2022exponential}
{\sc Z.~Zhang}, {\em Exponential convergence of {Sobolev} gradient descent for
  a class of nonlinear eigenproblems}, Comm. Math. Sci.,
  20 (2022), pp.~377--403.

\bibitem{zhou2003analysis}
{\sc A.~Zhou}, {\em An analysis of finite-dimensional approximations for the
  ground state solution of {Bose–Einstein} condensates}, Nonlinearity, 17
  (2003), pp.~541--550.

\bibitem{zhuang2019efficient}
{\sc Q.~Zhuang and J.~Shen}, {\em Efficient {SAV} approach for imaginary time
  gradient flows with applications to one- and multi-component {Bose-Einstein}
  condensates}, J. Comput. Phys., 396 (2019), pp.~72--88.
\end{thebibliography}
\end{document}